\documentclass[11pt, a4paper]{article}

\usepackage{pb-diagram, lamsarrow,pb-lams}
\usepackage{amssymb}
\usepackage{amsthm}
\usepackage[T1]{fontenc}
\usepackage{latexsym}
\usepackage{textcomp}
\usepackage[ansinew]{inputenc}
\usepackage{exscale}
\usepackage{graphicx}
\usepackage{enumerate}

    \newcommand{\cl}{\mathrm{cl}}

    \newcommand{\aug}{\mathrm{aug\,}}

    \newcommand{\irr}{\mathrm{Irr\,}}
    \newcommand{\irrp}{\mathrm{Irr_p\,}}

    \newcommand{\PMod}{\mathrm{PMod}}
    
    \newcommand{\perf}{\mathrm{perf}}
    \newcommand{\Dperf} {\mathcal D ^{\perf}}
    \newcommand{\ith}{^{\mathrm{th}}}
    
    \newcommand{\Ann}{\mathrm{Ann}}
    \newcommand{\Fitt}{\mathrm{Fitt}}
    
    \newcommand{\std}{\mathrm{std}}
    
    \newcommand{\Frob}{\mathrm{Frob}}
    \newcommand{\Spec}{\mathrm{Spec}}
    \newcommand{\nr}{\mathrm{nr}}
    \newcommand{\Gl}{\mathrm{Gl}}
    
    \newcommand{\Aff}{\mathrm{Aff}}

    \newcommand{\gp}{G_{\fP}}
    \newcommand{\ip}{I_{\fP}}
    \newcommand{\frob}{\phi_{\fP}}

    \newcommand{\ab}{^{\mathrm{ab}}}
    
    \newcommand{\punkt}{^{\bullet}}

    \newcommand{\zpmal} {\Z_p \mal}

    \newcommand{\zpg} {\zp G}
    \newcommand{\qpg} {\qp G}
    \newcommand{\cp} {\C_p}

    \newcommand{\arat} {\ar@}
    \newcommand{\inj} {\arat{^{(}->}}
    \newcommand{\sur} {\arat{->>}}
    \newcommand{\equal} {\arat{=}}
    \def\barat[#1] {\arat{->}'[#1][#1#1]}
    \def\binj[#1] {\inj'[#1][#1#1]}
    \def\bsur[#1] {\sur'[#1][#1#1]}
    \def\bequal[#1] {\equal'[#1][#1#1]}

    \newcommand{\ram}{_{\mathrm{ram}}}

    \newcommand{\hcl}{^{\mathrm{cl}}}

    \newcommand{\kot}[1]{K_{0}T({#1})}

    \newcommand{\beq}{\begin{equation}}
    \newcommand{\eeq}{\end{equation}}

    \newcommand{\mc}{\mathcal}

    \newcommand{\half}{{\textstyle \frac{1}{2}}}

     \newcommand{\fa}{\goth{a}}

     \newcommand{\fo}{\goth{o}}
     \newcommand{\fp}{\goth{p}}

     \newcommand{\fA}{\goth{A}}

     \newcommand{\fM}{\goth{M}}

     \newcommand{\fP}{\goth{P}}
     
     \newcommand{\fR}{\goth{R}}

     \newcommand{\C}{\mathbb{C}}

     \newcommand{\F}{\mathbb{F}}
     
     \newcommand{\N}{\mathbb{N}}
     \newcommand{\Q}{\mathbb{Q}}
     \newcommand{\R}{\mathbb{R}}

     \newcommand{\Z}{\mathbb{Z}}

    \newcommand{\ol}[1]{\overline{#1}}

    \newcommand{\ti}[1]{\tilde{#1}}
    
    \newcommand{\sm}{\setminus}
    \newcommand{\op}{\oplus}
    
    \newcommand{\me}{^{-1}}
    \newcommand{\mal}{^{\times}}

    \newcommand{\zg}{{\mathbb{Z}G}}
    \newcommand{\qg}{{\mathbb{Q}G}}
    
    \newcommand{\rg}{{\mathbb{R}G}}

    \newcommand{\zp}{{\mathbb{Z}_p}}
    
    \newcommand{\qp}{{\mathbb{Q}_p}}

    \newcommand{\into}{\rightarrowtail}
    
    \newcommand{\lto}{\longrightarrow}

    \newcommand{\ga}{\gamma}
    \newcommand{\Ga}{\Gamma}
    \newcommand{\si}{\sigma}
    
    \newcommand{\de}{\delta}
    \newcommand{\De}{\Delta}
    
    \newcommand{\la}{\lambda}
    \newcommand{\La}{\Lambda}
    
    \newcommand{\Om}{\Omega}
    \newcommand{\om}{\omega}
    \newcommand{\al}{\alpha}
    \newcommand{\ve}{\varepsilon}

    \newcommand{\bmp}{\begin{minipage}}
    \newcommand{\emp}{\end{minipage}}
    \newcommand{\btb}{\begin{tabular}}
    \newcommand{\etb}{\end{tabular}}
    \newcommand{\barr}{\begin{array}}
    \newcommand{\earr}{\end{array}}
    \newcommand{\bit}{\begin{itemize}}
    \newcommand{\eit}{\end{itemize}}
    \newcommand{\ben}{\begin{enumerate}}
    \newcommand{\een}{\end{enumerate}}
    \newcommand{\bct}{\begin{center}}
    \newcommand{\ect}{\end{center}}
    \newcommand{\bfr}{\begin{flushright}}
    \newcommand{\efr}{\end{flushright}}
    \newcommand{\bea}{\begin{eqnarray*}}
    \newcommand{\eea}{\end{eqnarray*}}
    \newcommand{\bqo}{\begin{quote}}
    \newcommand{\eqo}{\end{quote}}
    \newcommand{\bdc}{\begin{description}}
    \newcommand{\edc}{\end{description}}

    \newcommand{\Hom}{\mathrm{Hom}}

    \newcommand{\Ext}{\mathrm{Ext}}
    
    \newcommand{\Gal}{\mathrm{Gal}}

    \newcommand{\ind}{\mathrm{ind\,}}
    \newcommand{\infl}{\mathrm{infl\,}}
    
    \newcommand{\mod}{\mathrm{\ mod\ }}
    \newcommand{\res}{\mathrm{res\,}}
    \newcommand{\quot}{\mathrm{quot\,}}
    
    \newcommand{\Det}{\mathrm{Det\,}}

\input{amssym.def}
\input{amssym}
\input xy
\xyoption{all} \CompileMatrices

\newtheorem{theo}{Theorem}[section]
\newtheorem{prop}[theo]{Proposition}
\newtheorem{lem}[theo]{Lemma}
\newtheorem{cor}[theo]{Corollary}
\newtheorem{defi}[theo]{Definition}
\newtheorem{con}[theo]{Conjecture}
\newtheorem{rem}[theo]{Remark}

\newcommand{\myfootnote}[1]{%
\renewcommand{\thefootnote}{}%
\footnotetext{#1}%
\renewcommand{\thefootnote}{\arabic{footnote}}%
}

\begin{document}
\xymatrixrowsep{3pc} \xymatrixcolsep{3pc}

\title{Integrality of Stickelberger elements and the equivariant Tamagawa number conjecture}
\author{Andreas Nickel\thanks{I acknowledge financial support provided by the DFG}}
\date{}

\maketitle

\begin{abstract}
    Let $L/K$ be a finite Galois CM-extension of number fields with Galois group $G$.
    In an earlier paper, the author has defined a module $SKu(L/K)$ over the center
    of the group ring $\zg$ which coincides with the Sinnott-Kurihara ideal if $G$ is abelian and,
    in particular, contains many Stickelberger elements.
    It was shown that a certain conjecture on the integrality of $SKu(L/K)$ implies
    the minus part of the equivariant Tamagawa number conjecture at an odd prime $p$
    for an infinite class of (non-abelian) Galois CM-extensions of number fields which are at most tamely ramified
    above $p$, provided that Iwasawa's $\mu$-invariant vanishes. Here, we prove a relevant part of this
    integrality conjecture which enables us to deduce the minus-$p$-part of the equivariant Tamagawa number conjecture
    from the vanishing of $\mu$ for the same class of extensions. As an application we prove the non-abelian
    Brumer and Brumer-Stark conjecture outside the $2$-primary part for every monomial Galois extension of $\Q$
    provided that certain $\mu$-invariants vanish.
\end{abstract}

\section*{Introduction}
\myfootnote{{\it 2010 Mathematics Subject Classification} 11R42, 11R23}
\myfootnote{{\it Keywords:} Tamagawa number, equivariant $L$-values, Stickelberger elements}
Let $L/K$ be a finite Galois extension of number fields with Galois group $G$.
Burns \cite{Burns_equivariantI} used complexes arising from \'{e}tale cohomology of the
constant sheaf $\Z$ to define a canonical element $T\Om(L/K)$ of
the relative $K$-group $K_0(\zg, \R)$. This element relates the
leading terms at zero of Artin $L$-functions attached to $L/K$ to
natural arithmetic invariants.
It was shown that the vanishing of $T\Om(L/K)$
is equivalent to the equivariant Tamagawa number conjecture (ETNC) for the pair $(h^0(\Spec(L)), \zg)$ (cf.~\cite[Theorem 2.4.1]{Burns_equivariantI}).\\

The vanishing of $T\Om(L/K)$ is known to be true if $L$ is absolutely abelian
as proved by Burns and Greither \cite{Burns_Greither_ETNC}
with the exclusion of the $2$-primary part;  Flach \cite{Flach, Flach-2} extended the argument to cover the
$2$-primary part as well. Slightly weaker results in this cyclotomic case have been settled
independently by Ritter and Weiss \cite{towardI, representing}, Huber and Kings \cite{HK-Dirichlet}.
Some relatively abelian results are due to
Bley \cite{Bley_ETNC}; he showed that if $L/K$ is a finite abelian extension, where $K$ is an imaginary quadratic
field which has class number one, then the ETNC holds for all intermediate extensions $L/E$ such that $[L:E]$
is odd and divisible only by primes which split completely in $K/\Q$. Finally, if $L/K$ is a CM-extension and $p$ is odd,
the ETNC at $p$ naturally decomposes into a plus and a minus part; it was shown by the author \cite{ich-tame} that
the minus part of the ETNC at $p$ holds if $L/K$ is abelian and at most tamely ramified above $p$, and the Iwasawa $\mu$-invariant
vanishes if $p$ divides $|G|$ (and some additional technical condition is fulfilled). Note that the vanishing of $\mu$ is a long standing
conjecture of Iwasawa theory; the most general result is still due to Ferrero
and Washington \cite{Ferrero_Wash} and says that $\mu=0$ for absolutely abelian extensions.\\

For non-abelian extensions, the results are rather sparse.
Burns and Flach \cite{Burns_Flach_II} have given a proof for an infinite class of quaternion extensions over the rationals
and Navilarekallu \cite{Navilarekallu} has treated a specific $A_4$-extension over $\Q$.
Further progress has recently been made by Johnston and the author \cite{hybrid-ETNC}.

If $L/K$ is a CM extension,
the author \cite{ich-tameII} has introduced a module $SKu(L/K)$ over the center of the group ring $\zg$
which is a noncommutative analogue of the Sinnot-Kurihara ideal (cf.~\cite[p.~193]{Sinnott-ideal})
and was already implicitly used in \cite{ich-stark} and \cite{Burns_Johnston}. An integrality conjecture
on $SKu(L/K)$ has been formulated and it was shown that it is implied by the ETNC in many cases and follows from the results
in \cite{Barsky}, \cite{Cassou}, \cite{Deligne-Ribet} if $G$ is abelian.
Assuming the validity of this integrality conjecture, the minus part of the ETNC at $p$ was
deduced from the conjectural vanishing of $\mu$, provided that the ramification above $p$ is at most tame
(and, as in the abelian case, some technical extra assumption holds). Moreover, it follows from the results in
\cite{ich-stark} that for the case at hand the non-abelian analogues of Brumer's conjecture,
of the Brumer-Stark conjecture and of the strong Brumer-Stark property (as formulated in \cite{ich-stark}) hold,
provided that $\mu=0$ and the integrality conjecture holds.\\

Most of these results make heavily use of the validity of the equivariant Iwasawa main conjecture (EIMC) attached to the extension $L^+_{\infty} / K$,
where $L^+_{\infty}$ is the cyclotomic $\zp$-extension of $L^+$ which is the maximal real subfield of $L$.
Note that the EIMC is known for abelian extensions of totally real number fields with Galois group $\mc G$
such that $\mc G$ is a $p$-adic Lie group of dimension $1$ (cf.~\cite{Wiles-main,towardI}).
More recently, Ritter and Weiss \cite{EIMC-theorem} have shown that the EIMC (up to its uniqueness statement) holds
for arbitrary $p$-adic Lie groups of dimension $1$
provided that $\mu$ vanishes. In fact, this can be generalized to higher dimensional $p$-adic Lie groups
as shown by Kakde \cite{Kakde-mc} and, independently, by Burns \cite{Burns-mc}.
Note that Kakde in fact provides an independent proof also in the case of dimension $1$.\\

In this paper, we define a variant $SKu'(L/K)$ of the Sinnott-Kurihara module which is contained in
$SKu(L/K)$ and in fact equals $SKu(L/K)$ for abelian $G$. Let $\mc M(G)$ be a maximal order in $\qg$
containing $\zg$; for any ring $\La$, we write $\zeta(\La)$ for the subring of all elements which are central in $\La$.
The first main result is the following theorem which will be proved in \S \ref{sec:int-results}.
Recall that a character of a finite group is called \emph{monomial} if it is induced by a linear character of a subgroup.
A finite group is called monomial
if each of its (complex) irreducible characters is monomial.

\begin{theo} \label{SKu-MaxOrd}
Let $L/K$ be a Galois extension of number fields with Galois group $G$. If $G$ is monomial, then
$$SKu'(L/K) \subseteq \zeta(\mc M(G)).$$
\end{theo}

Now let $S$ and $T$ be two finite sets of places of $K$ such that $S$ and $T$ are disjoint and $S$ contains the set
$S_{\infty}$ of all infinite places of $K$. One can associate to $S$ and $T$ so-called Stickelberger elements
$\theta_S^T$ which lie in the center of the group ring algebra $\Q G$. These Stickelberger elements are defined via
values of Artin $L$-functions at zero and are closely related
to the Sinnott-Kurihara ideal $SKu(L/K)$; more precisely, they lie in $SKu(L/K)$ under suitable hypotheses on $S$ and $T$.
For instance, it suffices to assume the following hypotheses to which we will refer as $Hyp(S,T)$:
$S$ contains the set $S\ram$ of all ramified primes and no non-trivial root of unity in $L$
is congruent to $1$ modulo all primes $\fP \in T(L)$; here, for any set $T$ of places of $K$,
we write $T(L)$ for the set of places of $L$ which lie above those in $T$. \\

Now assume that the Galois group $G$ decomposes as $G = H \times C$, where $H$ is monomial and $C$ is abelian.
As before, let $\mc M(H)$ be a maximal order in $\Q H$ containing $\Z H$. Then we may view
$\mc M(H)[C]$ as an order in $\Q G$ and we have the following integrality statement for Stickelberger elements.

\begin{theo} \label{Stickelberger-int}
Let $L/K$ be a Galois extension of number fields with Galois group $G = H \times C$, where $H$ is monomial and $C$ is abelian.
Then $$\theta_S^T \in \zeta (\mc M(H)[C]) = \zeta(\mc M(H))[C]$$
whenever $Hyp(S,T)$ is satisfied.
\end{theo}

We will give a more precise statement and its proof in \S \ref{sec:int-results}. In fact, we will prove a more general result
involving also Stickelberger elements which are defined via values of Artin $L$-functions at negative integers.\\

Now let $p$ be a prime and let $L$ be either a totally real field or a CM-field. Let $\mc L$ be the cyclotomic $\zp$-extension
of $L(\zeta_p)$, where $\zeta_p$ denotes a primitive $p$-th root of unity. Moreover, let $X_{\std}$ be the `standard' Iwasawa
module which is the projective limit of the $p$-parts of the class groups in the cyclotomic tower of $L(\zeta_p)$.
We will say that the Iwasawa $\mu$-invariant attached to $L$ and $p$ vanishes if the $\mu$-invariant of $X_{\std}$ vanishes.
Now let $L/K$ be a Galois CM-extension of number fields with arbitrary Galois group $G$.
Then Theorem \ref{Stickelberger-int} is the key in proving our main result.

\begin{theo} \label{ETNC-minus}
Let $L/K$ be a Galois CM-extension of number fields with Galois group $G$ and let $p$ be a non-exceptional prime.
If the Iwasawa $\mu$-invariant attached to $L$ and $p$ vanishes, then
the $p$-minus part of the ETNC for the pair $(h^0(\Spec(L)), \Z G)$ is true.
\end{theo}

For a fixed extension $L/K$ there is only a finite number of exceptional primes; for a precise definition see \S \ref{sec:ETNC-almost}
(Definition \ref{defi:exceptional}),
where we will prove Theorem \ref{ETNC-minus}.
Finally, we obtain the following corollaries.

\begin{cor} \label{cor_int_conj}
Let $L/K$ be a Galois CM-extension of number fields with Galois group $G$ and let $p$ be a non-exceptional prime.
If the Iwasawa $\mu$-invariant attached to $L$ and $p$ vanishes,
then the $p$-part of the integrality conjecture (Conjecture \ref{integrality-conj} below) holds.
\end{cor}

\begin{cor} \label{cor-non-ab-stark}
Let $L/K$ be a Galois CM-extension of number fields with Galois group $G$ and let $p$ be a non-exceptional prime.
If the Iwasawa $\mu$-invariant attached to $L$ and $p$ vanishes, then the $p$-parts
of the following conjectures hold:
\ben[(i)]
    \item
    the non-abelian Brumer conjecture \cite[Conjecture 2.1]{ich-stark}
    \item
    the non-abelian Brumer-Stark conjecture \cite[Conjecture 2.6]{ich-stark}
    \item
    the minus part of the central conjecture (Conjecture 2.4.1) of Burns \cite{Burns-derivatives}
    \item
    the minus-part of the Lifted Root Number Conjecture of Gruenberg, Ritter and Weiss \cite{GRW}.
\een
Moreover, $L/K$ fulfills the non-abelian strong Brumer-Stark property  at $p$ (cf.~\cite[Definition 3.5]{ich-stark}).
\end{cor}

We will recall the precise statement of the non-abelian Brumer and Brumer-Stark conjecture in \S \ref{sec:Brumer-Stark}.

\begin{cor} \label{cor-etnc}
Let $L/K$ be a tamely ramified Galois CM-extension of number fields with Galois group $G$ and let $p$ be a non-exceptional prime.
If the Iwasawa $\mu$-invariant attached to $L$ and $p$ vanishes,
then the minus-$p$-parts of the central conjecture (Conjecture 3.3) of Breuning and Burns \cite{BB_at_1} and of
the ETNC for the pair $(h^0(\Spec(L))(1),\Z G)$ are valid.
\end{cor}

Finally, a further nice consequence of our results is the following theorem.

\begin{theo} \label{non-abelian-Stickelberger}
Let $L$ be any monomial Galois CM-extension of $\Q$.
Assume that the Iwasawa $\mu$-invariant attached to $L$ and $p$ vanishes for every odd prime $p$ which ramifies in $L$ or divides $[L:\Q]$.
Then the non-abelian Brumer conjecture and the non-abelian Brumer-Stark conjecture are true outside the $2$-primary parts.
\end{theo}

If $L$ is abelian over $\Q$, we know the vanishing of the $\mu$-invariants by the aforementioned result of Ferrero and Washington \cite{Ferrero_Wash}
and the above theorem recovers Stickelberger's classical theorem (cf.~\cite[Theorem 6.10]{Wash}). So the above result is a `non-abelian Stickelberger
theorem' without the extra factors occurring in \cite{Burns_Johnston}.

\subsubsection*{Acknowledgements}
I am very grateful to Jiro Nomura for informing me about his paper \cite{Nomura} and for drawing my attention
to monomial groups rather than just nilpotent groups. I am indebted to Henri Johnston for his hint how to avoid
assuming Leopoldt's conjecture in Corollary \ref{cor-etnc}, and his help with the example in \S \ref{sec:example}.

\section{Preliminaries} \label{sec:prelim}

\subsection{$K$-theory}
\subsubsection{Localization Sequences}
Let $\La$ be a left noetherian ring with $1$ and $\PMod(\La)$ the
category of all finitely generated projective $\La$-modules. We
write $K_0(\La)$ for the Grothendieck group of $\PMod(\La)$, and
$K_1(\La)$ for the Whitehead group of $\La$ which is the abelianized
infinite general linear group. If $S$ is a multiplicatively closed
subset of the center of $\La$ which contains no zero divisors, $1
\in S$, $0 \not \in S$, we denote the Grothendieck group of the
category of all finitely generated $S$-torsion $\La$-modules of
finite projective dimension by $K_0S(\La)$.
Writing $\La_S$ for the
ring of quotients of $\La$ with denominators in $S$,
we have the following
Localization Sequence (cf.~\cite[p.~65]{CR-II})
\begin{equation} \label{eqn:localization-sequence}
    K_1(\La) \to K_1(\La_S) \stackrel{\partial_S}{\lto} K_0S(\La) \to K_0(\La) \to K_0(\La_S).
\end{equation}
In the special case where $\La$ is an $\fo$-order over a commutative ring $\fo$ and $S$ is the set of all
nonzerodivisors of $\fo$, we also write $\kot\La$ instead of $K_0S(\La)$.
Moreover, we denote the relative $K$-group corresponding to
a ring homomorphism $\La \to \La'$ by $K_0(\La,\La')$ (cf.~\cite{Swan}).
Then we have a
Localization Sequence (cf.~\cite[p.~72]{CR-II})
\begin{equation} \label{eqn:localization-sequence-rel}
    K_1(\La) \to K_1(\La') \stackrel{\partial_{\La,\La'}}{\lto} K_0(\La,\La') \to K_0(\La) \to K_0(\La').
\end{equation}
The maps $\partial_S$ and $\partial_{\La,\La'}$ in (\ref{eqn:localization-sequence}) and (\ref{eqn:localization-sequence-rel})
are called \emph{boundary homomorphisms}.
It is also shown in \cite{Swan} that there is an isomorphism $K_0(\La,\La_S) \simeq K_0S(\La)$.

Let $G$ be a finite group; in the case where $\La'$ is the group ring $\rg$, the reduced norm map $\nr_{\rg}: K_1(\rg) \to \zeta(\rg)\mal$
is injective, and
there exists a canonical map $\hat \partial_G: \zeta(\rg)\mal \to K_0(\zg, \rg)$ such that the restriction
of $\hat \partial_G$ to the image of the reduced norm equals $\partial_{\zg, \rg} \circ \nr_{\rg}\me$.
This map is called the \emph{extended boundary homomorphism} and was introduced by Burns and Flach \cite{Burns_Flach}.

\subsubsection{Refined Euler characteristics}
For any ring $\La$ we write $\mathcal D (\La)$ for the derived
category of $\La$-modules. Let $\mathcal C^b (\PMod (\La))$ be the
category of bounded complexes of finitely generated projective
$\La$-modules. A complex of $\La$-modules is called perfect if it is
isomorphic in $\mathcal D (\La)$ to an element of $\mathcal C^b
(\PMod (\La))$. We denote the full triangulated subcategory of
$\mathcal D (\La)$ comprising perfect complexes by $\mathcal D
^{\perf} (\La)$. For any $C\punkt \in \mathcal C^b (\PMod (\La))$ we
define $\La$-modules
$$C^{ev} := \bigoplus_{i \in \Z} C^{2i},~ C^{odd} := \bigoplus_{i \in \Z} C^{2i+1}.$$
Similarly, we define $H^{ev}(C\punkt)$ and $H^{odd}(C\punkt)$ to be the direct sum over all even and odd degree
cohomology groups of $C\punkt$, respectively.

For the following let $R$ be a Dedekind domain of characteristic
$0$, $F$ its field of fractions, $A$ a finite dimensional
$F$-algebra and $\La$ an $R$-order in $A$. Let $K$ be a field containing $F$, and write
$K_0(\La, K)$ for $K_0(\La, K \otimes_F A)$. A pair $(C\punkt,t)$
consisting of a complex $C\punkt \in \mathcal \Dperf (\La)$ and an
isomorphism $t: H^{odd}(C_K \punkt) \to H^{ev}(C_K\punkt)$ is called a
trivialized complex, where $C_K\punkt := K \otimes^{\mathbb L}_R C\punkt$ is the
left derived tensor product of $C\punkt$
with $K$. We refer to $t$ as a trivialization of $C\punkt$.
One defines the refined Euler characteristic $\chi_{\La,K}
(C\punkt, t) \in K_0(\La,K)$ of a trivialized complex as follows:
Choose a complex $P\punkt \in \mathcal C^b(\PMod(\La))$ which is
quasi-isomorphic to $C\punkt$. Let $B^i(P_K \punkt)$ and $Z^i(P_K
\punkt)$ denote the $i\ith$ cobounderies and $i \ith$ cocycles of
$P_K \punkt$, respectively. We have the obvious exact sequences
$$ {B^i(P_K\punkt)} \into {Z^i(P_K\punkt)} \twoheadrightarrow {H^i(P_K\punkt)}\mbox{~,~~ }
   {Z^i(P_K\punkt)} \into {P_K^i} \twoheadrightarrow {B^{i+1}(P_K\punkt).}  $$
If we choose splittings of the above sequences, we get an
isomorphism
$$   \phi_t: P_K^{odd}  \simeq  \bigoplus_{i \in \Z} B^i(P_K\punkt) \op H^{odd}(P_K\punkt)
      \simeq  \bigoplus_{i \in \Z} B^i(P_K\punkt)  \op H^{ev}(P_K\punkt)
      \simeq  P_K^{ev},$$
where the second map is induced by $t$. Then the refined
Euler characteristic is defined to be
$$\chi_{\La, K} (C\punkt, t) := (P^{odd}, \phi_t, P^{ev}) \in K_0(\La, K)$$
which indeed is independent of all choices made in the
construction.
For further information concerning refined Euler characteristics
we refer the reader to \cite{Burns_Whitehead}.

We define $DT(\La)$ to be the torsion subgroup of $K_0(\La, F)$.

\subsubsection{$\Hom$-description} \label{sec:Hom-description}

In this paper we use a formulation of the ETNC in terms of relative $K$-theory and reduced norms. However,
we will frequently refer to \cite{ich-tame}, where the $\Hom$-description is used.
Here we summarize some basic facts of this equivalent theory for convenience of the reader.

Let $G$ be a finite group and let $p$ be a prime. We denote the ring of virtual characters of $G$
with values in $\cp$ by $R_p(G)$. We choose a finite Galois extension $E$ of $\qp$
such that all representations of $G$ can be realized over $E$, and put $\Ga := \Gal(E/\qp)$.
If $\chi$ is a character of $G$, we let $V_{\chi}$ be an $E G$-module with character $\chi$.
Then by \cite[Appendix A]{GRW} there is an isomorphism
\begin{eqnarray*}
  \Det: K_1(\qp G) & \stackrel{\simeq}{\lto} & \Hom_{\Ga}(R_p(G), E\mal) \\
  x & \mapsto & \left[ \chi \mapsto \det(X \mid \Hom_{E G}(V_{\chi}, (E G)^n)) \right],
\end{eqnarray*}
where $X \in \Gl_n(\qp G)$ maps to $x$ under the natural map $\Gl_n(\qp G) \to K_1(\qp G)$.
We have an exact sequence
\[
    K_1(\zp G) \rightarrow K_1(\qp G) \rightarrow K_{0}T(\zp G) \rightarrow 0,
\]
as the boundary homomorphism in the localization sequence (\ref{eqn:localization-sequence}) is surjective
in this case by a theorem of Swan \cite[Theorem 32.1]{CR-I}.
As $\zp G\mal$ surjects onto $K_1(\zp G)$ by \cite[Theorem 40.31]{CR-I}, we obtain the local $\Hom$-description
\[
    K_{0}T(\zp G) \simeq \Hom_{\Ga}(R_p(G), E\mal) / \Det(\zp G\mal).
\]
When $f \in \Hom_{\Ga}(R_p(G), E\mal)$ corresponds to $T \in K_{0}T(\zp G)$ under this isomorphism, then we
say that $f$ is a \emph{representing homomorphism} for $T$.
Similarly, by \cite[Theorem 45.3]{CR-II} the reduced norm induces isomorphisms
\begin{eqnarray*}
    \nr: K_1(\qp G) & \stackrel{\simeq}{\lto} & \zeta(\qp G)\mal \\
    K_{0}T(\zp G) & \simeq & \zeta(\qp G)\mal / \nr(\zp G\mal).
\end{eqnarray*}
We therefore have a commutative triangle
\begin{equation} \label{eqn:comm-triangle} \xymatrix{
    & K_1(\qp G) \ar[ld]_{\nr} \ar[rd]^{\Det} & \\
    {\zeta(\qp G)\mal} \ar[rr]^{f} & & \Hom_{\Ga}(R_p(G), E\mal)
}\end{equation}
in which each map is an isomorphism. We now describe the isomorphism $f$ in more detail.
For this let $z \in \zeta(\qp G)\mal$. Let $\irrp(G)$ be the set of absolutely irreducible
$\cp$-valued characters of $G$. For $\chi \in \irrp(G)$ we write $e_{\chi} := \frac{\chi(1)}{|G|} \sum_{g \in G} \chi(g\me) g$
for the associated central idempotent in $\cp G$. Note that $e_{\chi}$ actually belongs to $E G$ and each generates one of the
minimal ideals of $\zeta(E G)$; hence
\[
    \zeta(E G) = \bigoplus_{\chi \in \irrp(G)} E e_{\chi}.
\]
Consider $z$ as an element in $\zeta(E G)\mal$ and write $z = \sum_{\chi} z_{\chi} e_{\chi}$ with each $z_{\chi} \in E\mal$.
Then $f$ maps $z$ to the unique homomorphism $f_z \in \Hom_{\Ga}(R_p(G), E\mal)$ which on irreducible characters is given by
\begin{equation} \label{eqn:f_z}
    f_z: \chi \mapsto z_{\chi}, \quad \chi \in \irrp(G).
\end{equation}

\subsection{Equivariant $L$-values}
    Let us fix a finite Galois extension $L/K$ of number fields with Galois group $G$.
    For every prime $\fp$ of $K$ we fix a prime $\fP$ of $L$ above $\fp$ and write $\gp$ and $\ip$ for the decomposition
    group and inertia subgroup of $L/K$ at $\fP$, respectively. Moreover, we denote the residual group at $\fP$ by $\ol \gp = \gp / \ip$
    and choose a lift $\frob \in \gp$ of the Frobenius automorphism at $\fP$.

    \subsubsection{Complex $L$-series} \label{sec:complex-series}
    If $S$ is a finite set of places of $K$ containing the set $S_{\infty}$
    of all infinite places of $K$, and $\chi$ is a (complex) character of $G$, we denote the $S$-truncated Artin $L$-function
    attached to $\chi$ and $S$ by $L_S(s,\chi)$ and define $L_S^{\ast}(r,\chi)$ to be the leading coefficient
    of the Taylor expansion of $L_S(s,\chi)$ at $s=r$, $r \in \Z_{\leq 0}$. Recall that there is a canonical isomorphism
    $\zeta(\C G) = \prod_{\chi \in \irr (G)} \C$, where $\irr (G)$ denotes the set
    of complex valued irreducible characters of $G$. We define the equivariant Artin $L$-function to be the
    meromorphic $\zeta(\C G)$-valued function
    $$L_S(s) := (L_S(s,\chi))_{\chi \in \irr (G)}.$$
    We put $L_S^{\ast}(r) = (L_S^{\ast}(r,\chi))_{\chi \in \irr (G)}$ and abbreviate $L_{S_{\infty}}(s)$ by $L(s)$.
    If $T$ is a second finite set of places of $K$ such that $S \cap T = \emptyset$, we define
    $\de_T(s) := (\de_T(s,\chi))_{\chi\in \irr (G)}$, where $\de_T(s,\chi) = \prod_{\fp \in T} \det(1 - N(\fp)^{1-s} \frob\me| V_{\chi}^{\ip})$
    and $V_{\chi}$ is a $G$-module with character $\chi$.
    We put
    $$\Theta_{S,T}(s) := \de_T(s) \cdot L_S(s)^{\sharp},$$
    where we denote by $^{\sharp}: \C G \to \C G$ the anti-involution induced by $g \mapsto g\me$.
    These functions are the so-called $(S,T)$-modified $G$-equivariant $L$-functions, and for every integer $r \leq 0$ we define Stickelberger elements
    $$\theta_S^T(r) := \Theta_{S,T}(r) \in \zeta(\C G).$$
    For convenience, we also put $L_S^T(s,\chi) := \de_T(s,\check \chi) \cdot L_S(s,\chi)$,
    where we write $\check \chi$ for the character contragredient to $\chi$. Thus
    $$\theta_S^T(r)^{\sharp} = (L_S^T(r,\chi))_{\chi \in \irr (G)}.$$
    We will also write $L_S^T(L/K,s,\chi)$ for $L_S^T(s,\chi)$ if the extension $L/K$ is not clear from the context,
    and similarly for $\theta_S^T(r)$.
    If $T$ is empty, we abbreviate $\theta_S^T(r)$ by $\theta_S(r)$.
    Now a result of Siegel \cite{Siegel} implies that
    \beq \label{Stickelberger_is _rational}
        \theta_S^T(r) \in \zeta(\Q G)
    \eeq
    for all $r \leq 0$.
    Let us fix an embedding $\iota: \C \into \cp$; then
    the image of $\theta_S^T$ in $\zeta(\qp G)$ via the canonical embedding
    $$\zeta(\Q G) \into \zeta (\qp G) = \bigoplus_{\chi \in \irrp(G) / \sim} \qp (\chi),$$
    is given by $\sum_{\chi} L_S^T(r, \check \chi^{\iota\me})^{\iota}$.
    Here, the sum runs over all $\cp$-valued irreducible characters of $G$ modulo Galois action.
    Note that we will frequently drop $\iota$ and $\iota\me$ from the notation. Finally, for an irreducible character $\chi$
    with values in either $\C$ or $\cp$ we put
    $e_{\chi} = \frac{\chi(1)}{|G|} \sum_{g \in G} \chi(g\me) g$ which is a central idempotent in either $\C G$ or $\cp G$.

    \subsubsection{$p$-adic $L$-series}
    Now let $L/K$ be a Galois CM-extension, i.e.~$L$ is a CM-field,
    $K$ is totally real and complex conjugation induces a unique automorphism $j$ of $L$ which lies in the center of $G$.
    Recall that a character $\chi$ of $G$ is called \emph{even} when $\chi(j) = \chi(1)$, and it is called \emph{odd} when $\chi(j) = -\chi(1)$.
    Fix an odd prime $p$ and suppose that $S$ also contains all $p$-adic places of $K$.
    Let $\chi$ be an even character of $G$ and denote the $S$-truncated $p$-adic Artin $L$-series of $\chi$ by
    $L_{p,S}(s,\chi)$. Then for every integer $r \geq 2$ one has
    \begin{equation} \label{eqn:interpolation-property}
        L_{p,S}(1-r, \chi) = L_S(1-r, \chi \omega^{-r}),
    \end{equation}
    where $\omega$ denotes the Teichm\"{u}ller character.
    When $\chi$ is a linear character, then the interpolation property (\ref{eqn:interpolation-property})
    for every $r \geq 1$ follows from the work of Deligne and Ribet \cite{Deligne-Ribet}. The general case
    for $r \geq 2$ is then established by using Serre's variant of Brauer induction (see \cite[Chapitre III, \S 1]{Tate-Stark}).
    In the case $r=1$, however, this argument fails due to the potential presence of trivial zeros of the $p$-adic $L$-series of $\chi$
    at zero. One nevertheless suspects that the identity
    \begin{equation} \label{eqn:values-padic-complex}
        L_{p,S}(0, \chi) = L_S(0, \chi \omega^{-1})
    \end{equation}
    holds in general. As both sides behave well under direct sum, inflation and induction of characters, we see that
    (\ref{eqn:values-padic-complex}) at least holds when $\chi$ is a monomial character
    (see the discussion in \cite[\S 2]{Gross-p-adic}).

    Note that the identity (\ref{eqn:values-padic-complex}) is implicitly assumed to hold in \cite{ich-tameII}
    and \cite[\S 4]{Nickel-IwasawaStark}. This will not affect our main Theorem \ref{ETNC-minus},
    as we will first reduce to monomial Galois groups. However, such a reduction step is not possible
    for the Brumer-Stark conjecture. As a consequence we have to restrict to monomial extensions
    in Theorem \ref{non-abelian-Stickelberger}.

    \subsection{Ray class groups}
    Let $T$ and $S$ be finite sets of places as in \S \ref{sec:complex-series};
    so $S$ contains all infinite places and $S \cap T = \emptyset$. We write $\cl_L^T$ for the ray class group of $L$ to the ray
    $\fM_T := \prod_{\fP \in T(L)} \fP$ and $\mc{O}_S$ for the ring of $S(L)$-integers of $L$.
    We denote the $S(L)$-units of $L$ by $E_S$ and
    define
    $E_S^T := \left\{x \in E_S: x \equiv 1 \mod \fM_T \right\}$.
    If $S = S_{\infty}$, we also write $E_L^T$ for $E_{S_{\infty}}^T$.
    All these modules are equipped with a
    natural $G$-action.
    Now suppose that $L/K$ is a Galois CM-extension and let $j \in G$ denote complex conjugation.
    If $R$ is a subring of either $\C$ or $\cp$ for a prime $p$ such that $2$ is invertible in $R$,
    we put $R G_+ := R G / (1 - j)$ and $R G_- := R G / (1 + j)$ which are rings, since the idempotents $(1\pm j)/2$ lie in $R G$.
    For any $R G$-module $M$ we define $M^+ = R G_+ \otimes_{R G} M$ and $M^- = R G_- \otimes_{R G} M$ which are exact functors
    since $2 \in R\mal$. We define
    $$A_L^T  := (\Z[\half] \otimes_{\Z} \cl_L^T)^-.$$
    If $M$ is a finitely generated
    $\Z$-module and $p$ is a prime, we put $M(p) := \zp \otimes_{\Z} M$. For odd primes $p$,
    we will in particular consider $A_L^T(p)$, the minus $p$-part of the ray class group $\cl_L^T$.

\subsection{Noncommutative Fitting invariants}
For the following we refer the reader to \cite{ich-Fitting} and \cite{JN_Fitting}.
We denote the set of all $m \times n$
matrices with entries in a ring $R$ by $M_{m \times n} (R)$ and in the case $m=n$
the group of all invertible elements of $M_{n \times n} (R)$ by $\Gl_n(R)$.

\subsubsection{$\nr(\La)$-equivalence}
Let $A$ be a separable $K$-algebra and $\La$ be an $\fo$-order in $A$, finitely generated as $\fo$-module,
where $\fo$ is an integrally closed complete commutative noetherian local domain with field of quotients $K$.
The group ring $\zpg$ of a finite group $G$ will serve as a standard example.
Let $N$ and $M$ be two $\zeta(\La)$-submodules of
    an $\fo$-torsionfree $\zeta(\La)$-module.
    Then $N$ and $M$ are called {\it $\nr(\La)$-equivalent} if
    there exists an integer $n$ and a matrix $U \in \Gl_n(\La)$
    such that $N = \nr(U) \cdot M$, where $\nr: A \to \zeta(A)$ denotes
    the reduced norm map which extends to matrix rings over $A$ in the obvious way.
    We denote the corresponding equivalence class by $[N]_{\nr(\La)}$.
    We say that $N$ is
    $\nr(\La)$-contained in $M$ (and write $[N]_{\nr(\La)} \subseteq [M]_{\nr(\La)}$)
    if for all $N' \in [N]_{\nr(\La)}$ there exists $M' \in [M]_{\nr(\La)}$
    such that $N' \subseteq M'$. Note that it suffices to check this property for one $N_0 \in [N]_{\nr(\La)}$.
    We will say that $x$ is contained in $[N]_{\nr(\La)}$ (and write $x \in [N]_{\nr(\La)}$) if there is $N_0 \in [N]_{\nr(\La)}$ such that $x \in N_0$.\\

\subsubsection{Noncommutative Fitting invariants}
    Now let $M$ be a finitely presented (left) $\La$-module and let
    \beq \label{finite_representation}
        \La^a \stackrel{h}{\lto} \La^b \twoheadrightarrow M
    \eeq
    be a finite presentation of $M$.
    We identify the homomorphism $h$ with the corresponding matrix in $M_{a \times b}(\La)$ and define
    $S(h) = S_b(h)$ to be the set of all $b \times b$ submatrices of $h$ if $a \geq b$. In the case $a=b$
    we call (\ref{finite_representation}) a \emph{quadratic presentation}.
    The Fitting invariant of $h$ over $\La$ is defined to be
    $$\Fitt_{\La}(h) = \left\{ \barr{lll} [0]_{\nr(\La)} & \mbox{ if } & a<b \\
                        \left[\langle \nr(H) | H \in S(h)\rangle_{\zeta(\La)}\right]_{\nr(\La)} & \mbox{ if } & a \geq b. \earr \right.$$
    We call $\Fitt_{\La}(h)$ a \emph{Fitting invariant} of $M$ over $\La$. One defines $\Fitt_{\La}^{\max}(M)$ to be the unique
    Fitting invariant of $M$ over $\La$ which is maximal among all Fitting invariants of $M$ with respect to the partial
    order ``$\subseteq$''. If $M$ admits a quadratic presentation $h$, one also puts $\Fitt_{\La}(M) := \Fitt_{\La}(h)$
    which is independent of the chosen quadratic presentation.

    \begin{rem} \label{rem:rephom-vs-fittgen}
      Suppose that $\La = \zp G$ for a finite group $G$ and that $M$ is a finite $\zp G$-module with projective dimension at most $1$.
      Then $M$ admits a quadratic presentation $h$ which must be injective. The corresponding matrix $H$ then belongs to $\Gl_a(\qp G)$
      and $z := \nr (H)$ is a generator of the Fitting invariant $\Fitt_{\zp G}(M)$. Using the $\Hom$-description of \S \ref{sec:Hom-description},
      the commutative triangle (\ref{eqn:comm-triangle}) shows that the homomorphism $f_z$ defined in (\ref{eqn:f_z}) is a representing
      homomorphism of the class of $M$ in $K_0T(\zp G)$. Conversely, if $f$ is a representing homomorphism for the class of $M$,
      then there is $z \in \zeta(\qp G)\mal$ such that $f = f_z$, and $z$ generates the Fitting invariant of $M$.
    \end{rem}

    \subsubsection{Denominator ideals and the integrality ring}
    Assume now that $\fo$ is an integrally closed commutative noetherian domain, but not necessarily complete or local.
    We denote by $\mc I(\La)$ the $\zeta(\La)$-submodule of $\zeta(A)$ generated by the
    elements $\nr(H)$, $H \in M_{b\times b}(\La)$, $b \in \N$. As the reduced norm is multiplicative,
    we see that $\mc I(\La)$ is in fact a commutative ring which we call the \emph{integrality ring} of $\La$.

    We may decompose the separable $K$-algebra $A$ into its simple components
    $$A = A_1 \op \cdots \op A_t,$$
    i.e.~each $A_i$ is a simple $K$-algebra and $A_i = A e_i = e_i A$ with central primitive idempotents $e_i$, $1 \leq i \leq t$.
    Each $A_i$ is isomorphic to an algebra of $n_i \times n_i$ matrices over a skewfield $D_i$ and $K_i = \zeta(D_i)$ is a finite
    field extension of $K$. We denote the Schur index of $D_i$ by $s_i$ such that $[D_i:K_i] = s_i^2$.
    We choose a maximal order $\La'$ containing $\La$. Then also $\La'$ decomposes into $\La' = \op_{i=1}^t \La'_i$,
    where $\La'_i = \La' e_i$.
    Now let $H \in M_{b \times b} (\La)$ and write $H = \sum_{i=1}^t H_i$, where each $H_i$ is a $b \times b$ matrix with entries in
    $\La'_i$. Let $m_i = n_i \cdot s_i \cdot b$ and let $f_i(X) = \sum_{j=0}^{m_i} \al_{ij}X^j$ be the reduced characteristic polynomial of $H_i$.
    We put
    $$H_{i}^{\ast} := (-1)^{m_{i}+1} \cdot \sum_{j=1}^{m_i} \alpha_{ij}H_{i}^{j-1}, \quad H^{\ast} := \sum_{i=1}^{t} H_{i}^{\ast}.$$
    Then by \cite[Lemma 3.4]{JN_Fitting}, we have $H^{\ast} \in M_{b \times b}(\La')$ and
    $H^{\ast} H = H H^{\ast} = \nr (H) \cdot 1_{b \times b}$; note that the condition on $\fo$ to be a complete local ring is not necessary for this result.
    Moreover, we point out that this definition follows \cite{JN_Fitting}, but slightly differs from the corresponding notion in \cite{ich-Fitting}.
    If  $\ti H \in M_{b \times b} (\La)$ is a second matrix, then $(H \ti H)^{\ast} = \ti H^{\ast} H^{\ast}$.
    We define
    $$\mathcal H(\La) := \left\{
        x \in \zeta(\La) |  x H^{\ast} \in M_{b \times b}(\La)  \forall b \in \N ~\forall H \in M_{b \times b} (\La)
         \right\}.$$
    Since $x \cdot \nr(H) = x H^{\ast} H$, we have in particular
    \beq \label{HI_in_zeta}
        \mathcal H(\La) \cdot \mathcal I(\La) = \mc H(\La) \subseteq \zeta(\La)
    \eeq
    and we call $\mc H(\La)$ the \emph{denominator ideal} of $\La$.
    For convenience, we put $\mc H_p(G) := \mc H(\zpg)$ and $\mc H(G) := \mc H(\Z G)$, and similarly
    $\mc I_p(G) := \mc I(\zpg)$ and $\mc I(G) := \mc I(\Z G)$.
    The importance of the denominator ideal $\mathcal H(\La)$ will become clear by means of the following result
    which is \cite[Theorem 3.6]{JN_Fitting} (see also \cite[Theorem 4.2]{ich-Fitting}).
    \begin{theo} \label{annihilation-theo}
        If $\fo$ is an integrally closed complete commutative noetherian local domain
        and  $M$ is a finitely presented $\La$-module, then
        $$\mathcal H(\La) \cdot \Fitt_{\La}^{\max}(M) \subseteq \Ann_{\La}(M).$$
    \end{theo}

    We will need the following lemma whose last claim is \cite[Lemma 6.6]{ich-tameII}.
    \begin{lem} \label{int-lem}
      Let $G$ and $C$ be finite groups with $C$ abelian.
      Let $p$ be a prime and choose a maximal order $\mc M_p(G)$ in $\qp G$ which contains $\zp G$.
      Then we have inclusions
      \[
      \mc I_p(G)[C] \subseteq \mc I_p(G \times C) \subseteq \zeta(\mc M_p(G))[C].
      \]
      In particular
      we have $|G| \cdot \mc I_p(G \times C) \subseteq \zeta(\zp [G \times C])$
      for all primes $p$.
    \end{lem}
    \begin{proof}
      The ring $\mathcal{I}_{p}(G)[C]$ (resp.~$\mathcal{I}_{p}(G \times C)$) is generated over
      $\zeta(\Z_{p}[G])[C] = \zeta(\Z_{p}[G \times C])$ by the elements $\nr(H)$,
      where $H$ runs through $M_{n \times n}(\zp G)$ (resp.~$M_{n \times n}(\Z_p[G \times C])$), $n \in \N$.
      So we have $\mathcal{I}_{p}(G)[C] \subseteq \mathcal{I}_{p}(G \times C)$.
      The proof of the second inclusion is essentially the same as that of \cite[Lemma 6.6]{ich-tameII}; we include it here for convenience.
      Up to conjugation the maximal order $\mc M_p(G)$ is a direct sum of matrix rings of type
      $M_{n \times n} (\mc{O}_D)$, where $\mc{O}_D$ denotes the valuation ring of a skew field $D$.
      Note that conjugation does neither change the center of the order nor the image of the reduced norm.
      We have
      $$\zeta(M_{n \times n} (\mc{O}_D)) = \zeta(\mc{O}_D) = \mc{O}_F,$$
      where $\mc{O}_F$ is the ring of integers of the field $F = \zeta(D)$ which is finite over $\qp$.
      Since the reduced norm maps $\mc M_p(G)$ into its center and $|G| \cdot \zeta(\mc M_p(G)) \subseteq \zeta(\zpg)$,
      it suffices to show that the reduced norm maps $M_{m \times m}(M_{n \times n} (\mc{O}_D) [C])$ into $\mc{O}_F[C]$.
      Let us first assume that $D = F$. Then the map
      \bea
        \si: M_{n \times n}(F)[C] & \lto & M_{n \times n}(F[C])\\
        \sum_{c \in C} M_c c & \mapsto & (\sum_{c \in C} \al_{ij}(c) c)_{i,j}
      \eea
      is an isomorphism of rings, where $M_c = (\al_{ij}(c))_{i,j}$ lies in $M_{n \times n}(F)$. Likewise, $\si$ induces an isomorphism
      $$\si: M_{n\times n}(\mc{O}_F) [C] \simeq M_{n \times n}(\mc{O}_F [C]).$$
      Therefore, we have
      $$\nr(M_{m\times m}(M_{n \times n}(\mc{O}_F)[C])) = \nr(M_{nm \times nm}(\mc{O}_F[C])) = \nr(\mc{O}_F[C]) = \mc{O}_F[C].$$
      For arbitrary $D$, there is a field $E$, Galois over $F$ such that $E \otimes_F D \simeq M_{s \times s} (E)$
      for some integer $s$. We have just proven that the reduced norm maps
      $M_{m \times m}(M_{n \times n} (\mc{O}_D) [C])$ into $\mc{O}_E[C]$. But the image is invariant under
      the action of $Gal(E/F)$
      and is therefore contained in $\mc{O}_F[C]$.
    \end{proof}

    \begin{cor} \label{cor:I-under-abelian-product}
      Let $p$ be a prime and let $G$ and $C$ be finite groups with $C$ abelian.
      If $\mc I_p(G) = \zeta(\mc M_p(G))$, then $\mc I_p(G \times C) = \mc I_p(G)[C] = \zeta(\mc M_p(G))[C]$.
    \end{cor}

    The following result determines all primes $p$ for which the denominator ideal $\mc H_p(G)$ is best possible.

    \begin{prop} \label{prop:best-denominators}
      We have $\mc H_p(G) = \zeta(\zpg)$ if and only if $p$ does not divide the order of the commutator subgroup $G'$ of $G$.
      Furthermore, when this is the case we have $\mc I_p(G) = \zeta(\zpg)$.
    \end{prop}

    \begin{proof}
      The first assertion is a special case of \cite[Proposition 4.8]{JN_Fitting}. The second assertion then follows from (\ref{HI_in_zeta}).
    \end{proof}

    \section{The integrality conjectures} \label{sec:int-conjectures}
    Let $L/K$ be a Galois extension with Galois group $G$.
    Let $S$ and $T$ be two finite sets of places of $K$ such that
    \ben[(i)]
        \item
        $S$ contains all the infinite places of $K$ and all the places which ramify in $L/K$,
        i.e.~$S \supseteq S_{\ram} \cup S_{\infty}$.
        \item
        $S \cap T = \emptyset$.
        \item
        $E_S^T$ is torsionfree.
    \een
    We refer to the above hypotheses as $Hyp(S,T)$.
    For a fixed set $S$ we define $\fA_S$ to be the $\zeta(\zg)$-submodule of $\zeta(\qg)$ generated
    by the  elements $\de_T(0)$, where $T$ runs through the finite sets of places of $K$
    such that $Hyp(S,T)$ is satisfied. Note that $\fA_S$ equals the $\zg$-annihilator of the roots
    of unity of $L$ if $G$ is abelian by \cite[Lemma 1.1, p.~82]{Tate-Stark}.

    \subsection{The Sinnott-Kurihara ideal}
    For any finite group $H$ we put $N_H := \sum_{h \in H} h$.
    For a finite prime $\fp$ of $K$, we define a $\Z\gp$-module $U_{\fp}$ by
    $$U_{\fp} := \langle N_{\ip}, 1 - \ve_{\fp}\frob\me\rangle_{\Z \gp} \subset \Q \gp,$$
    where $\ve_{\fp} = |\ip|\me N_{\ip}$.
    Note that $U_{\fp} = \Z \gp$ if $\fp$ is unramified in $L/K$ such that the definition of the
    following $\mc I(G)$-module is indeed independent of the set $S$ as long as $S$ contains the ramified primes:
    $$U := \langle \prod_{\fp \in S \sm S_{\infty}} \nr(u_{\fp}) | u_{\fp} \in U_{\fp} \rangle_{\mc I(G)} \subset \zeta(\qg).$$
    \begin{defi}
        Let $S$ be a finite set of primes which contains $S\ram \cup S_{\infty}$. We define an $\mc I(G)$-module by
        $$SKu(L/K, S) := \fA_S \cdot U \cdot L(0)^{\sharp} \subset \zeta(\qg).$$
        We call $SKu(L/K) := SKu(L/K, S\ram \cup S_{\infty})$ the (fractional) \textbf{Sinnott-Kurihara ideal}.
    \end{defi}
    For abelian $G$, this definition coincides with the Sinnott-Kurihara ideal $SKu(L/K)$
    in \cite{Gr-Fitt-ETNC} (see also \cite[p.~193]{Sinnott-ideal}) and is closely related to the Fitting ideal of the Pontryagin dual of the class group
    (see \cite[Theorem 8.8]{Gr-Fitt-ETNC}).

    Note that our definition slightly differs from the original
    definition in \cite{ich-tameII}, where in the definition of $U$ the integrality ring $\mc I(G)$
    is replaced with $\zeta(\zg)$. However, as observed by the reviewer, it is then not clear whether the
    definition of $U$ does not depend on $S$. We assure the reader that this redefinition does not affect
    any of the results in \cite{ich-tameII}.

    The integrality conjecture as formulated in \cite{ich-tameII} (where $L/K$
    is assumed to be a CM-extension; but we will not assume this here) now asserts the following:
    \begin{con} \label{integrality-conj}
        The Sinnott-Kurihara ideal $SKu(L/K)$ is contained in $\mc I(G)$.
    \end{con}

    \begin{rem}
      \ben[(i)]
            \item
            Since clearly $SKu(L/K,S) \subseteq SKu(L/K,S')$ whenever $S' \subseteq S$, Conjecture \ref{integrality-conj}
            implies $SKu(L/K,S) \subseteq \mc I(G)$ for all admissible sets $S$.
            \item
            If the sets $S$ and $T$ satisfy $Hyp(S,T)$, the Stickelberger element
            $\theta_S^T(0)$ is contained in $SKu(L/K,S)$. Hence Conjecture \ref{integrality-conj}
            predicts that $\theta_S^T(0) \in \mc I(G)$ which is part of \cite[Conjecture 2.1]{ich-stark}.
            \item
            In the above definitions, we may replace $\Z$ and $\Q$ by $\zp$ and $\qp$, respectively.
            We obtain a local Sinnott-Kurihara ideal $SKu_p(L/K)$ contained in $\zeta(\qpg)$. Since we have an equality
            $$\mc I(G) = \bigcap_p (\mc I_p(G) \cap \zeta(\qg)),$$
            we have an equivalence
            $$SKu(L/K) \subseteq \mc I(G) \iff SKu_p(L/K) \subseteq \mc I_p(G) ~\forall p.$$
      \een
    \end{rem}

    If $G$ is abelian, we obviously have $\mc I(G) = \zeta(\zg) = \zg$ and the results in
    \cite{Barsky}, \cite{Cassou}, \cite{Deligne-Ribet} each imply the following theorem (cf.~\cite[\S 2]{Gr-Fitt-ETNC}).

    \begin{theo} \label{BCDR-theo-0}
      Conjecture \ref{integrality-conj} holds if $L/K$ is an abelian extension.
    \end{theo}

    \subsection{The modified Sinnott-Kurihara ideal}
    We also define a modified version of the Sinnott-Kurihara ideal as follows. For a finite prime $\fp$ of $K$, define a $\mc I(\gp)$-module
    $U_{\fp}'$ by
    $$U_{\fp}' := \langle \nr(N_{\ip}), \nr(1 - \ve_{\fp}\frob\me)\rangle_{\mc I(\gp)} \subset \zeta(\Q \gp).$$
    If $S$ contains $S\ram \cup S_{\infty}$, we define
    $$U' := \prod_{\fp \in S \sm S_{\infty}} U_{\fp}',$$
    $$SKu'(L/K,S) := \fA_S \cdot U' \cdot L(0)^{\sharp} \subseteq SKu(L/K,S).$$
    As before, the definition of $U'$ does not depend on $S$ and
    all the above remarks remain true if we replace $SKu(L/K,S)$ by $SKu'(L/K,S)$ throughout.
    We put $SKu'(L/K) := SKu'(L/K, S\ram \cup S_{\infty})$. If $G$ is abelian, the reduced norm
    is just the identity on $\Q G$. Moreover, $\fA_S$ is the whole $\Z G$-annihilator of $\mu_L$, the roots of unity in $L$,
    and hence independent of $S$. This implies the following proposition.

    \begin{prop}
    If $L/K$ is an abelian extension, then
    $$SKu(L/K) = SKu(L/K,S) = SKu'(L/K,S) \subseteq \Z G$$
    for all admissible sets $S$.
    \end{prop}

    \subsection{Negative integers}
    We now discuss a (partial) analogue of Conjecture \ref{integrality-conj} in the case, where $r<0$ is a negative integer.
    We denote the absolute Galois group of $L$
    by $G_L$ and put $\mu_{1-r}(L) := (\Q / \Z)(1-r)^{G_L}$, where $(\Q / \Z)(1-r)$ denotes the usual
    $(1-r)$-fold Tate twist of $\Q/\Z$.

    \begin{con} \label{integrality-conj_r}
    Let $L/K$ be a Galois extension of number fields with Galois group $G$ and let $r<0$ be an integer.
    Then for every $x \in \Ann_{\zg}(\mu_{1-r}(L))$ one has
    $$\nr(x) \cdot \theta_S(r) \in \mc I(G)$$
    for all finite sets $S$ of primes of $K$ containing $S\ram \cup S_{\infty}$.
    \end{con}

    \begin{rem}
    \ben[(i)]
        \item
        If $S \subseteq S'$, then we have an equality
        $$\theta_{S'}(r) = (\prod_{\fp \in S' \sm S} \nr(1-N(\fp)^{-r} \frob\me)) \theta_S(r).$$
        As $\nr(1-N(\fp)^{-r} \frob\me))$ lies in $\mc I(G)$, it suffices to consider the set $S = S\ram \cup S_{\infty}$
        in Conjecture \ref{integrality-conj_r}.
        \item
        Since the $\zg$-annihilator of $\mu_{1-r}(L)$ is generated by the elements $\prod_{\fp \in T} (1 - \frob N(\fp)^{1-r})$, where $T$ runs through
        all finite sets of primes in $K$ such that $Hyp(S,T)$ is satisfied (cf.~\cite{Coates-annihilator}),
        Conjecture \ref{integrality-conj_r} in particular implies that $\theta_S^T(r) \in \mc I(G)$ for all
        finite sets of primes $S$ and $T$ such that $Hyp(S,T)$ holds.
        \item
        Note that Conjecture \ref{integrality-conj_r} outside its $2$-primary part implicitly is a part of \cite[Conjecture 2.11]{ich-K-groups} if
        either $L/K$ is a CM-extension and $r$ is even or $L/K$ is an extension of totally real fields and $r$ is odd.
    \een
    \end{rem}

    Again, the results in \cite{Barsky}, \cite{Cassou}, \cite{Deligne-Ribet} each imply the following theorem.

    \begin{theo} \label{BCDR-theo-r}
    Conjecture \ref{integrality-conj_r} holds if $L/K$ is an abelian extension.
    \end{theo}

    \section{A reduction step} \label{sec:reduction-step}

    In order to prove one of the conjectures of the preceding paragraph, we may henceforth assume that the field $K$ is totally real,
    as otherwise $\theta_{S_{\infty}}(r) = L(r)^{\sharp} = 0$; hence also $SKu(L/K) = 0$ and $\theta_S(r) = 0$ for all finite sets $S$
    containing $S_{\infty}$. By the same reason, we may assume that $L$ is totally complex if $r$ is even.
    Note that we actually have to exclude the special case, where $r=0$ and $L / \Q$ is a CM-extension of degree $2$.
    But in this case the occurring Galois group is abelian and everything is known by Theorem \ref{BCDR-theo-0}.\\

    Let us denote the set of complex places of $L$ by $S_{\C}(L)$. For every $w \in S_{\C}(L)$, the decomposition group $G_w$ is cyclic of order
    two and we denote its generator by $j_w$. If $r$ is even we define
    $$H = H(r) := \langle j_w \cdot j_{w'} \mid w, w' \in S_{\C}(L) \rangle.$$
    If r is odd, we define
    $$H = H(r) := \langle j_w  \mid w \in S_{\C}(L) \rangle.$$
    In both cases, $H$ is normal in $G$ such that the fixed field $L^H$ is a Galois extension of $K$ with Galois group $\ol G := G/H$.
    Note that $L^H / K$ is either a Galois CM-extension or a Galois extension of totally real fields.

    \begin{prop} \label{reduction-step}
    Let $L/K$ be a Galois extension of number fields with Galois group $G$ and let $p$ be an odd prime.
    Assume that $L$ is totally imaginary if we consider Conjecture \ref{integrality-conj} or Conjecture \ref{integrality-conj_r} for even $r$.
    Assume further that $G$ has a unique $2$-Sylow subgroup. Then
    the $p$-part of Conjecture \ref{integrality-conj} (resp.~Conjecture \ref{integrality-conj_r}) is true for $L/K$ if and only if the $p$-part of
    Conjecture \ref{integrality-conj} (resp.~Conjecture \ref{integrality-conj_r}) is true for $L^H / K$.
    \end{prop}

    \begin{proof}
    Since $H$ is normal in $G$, the group ring element $\ve_H := |H|\me N_H$ is a central idempotent in $\qp G$.
    Let $G_2$ be the unique $2$-Sylow subgroup of $G$. Then $j_w$ lies in $G_2$ for every $w \in S_{\C}(L)$ such that $H$ is a finite $2$-group.
    Since $p$ is odd, this implies that $\ve_H$ actually lies in $\zp G$. Now $\zp G$ decomposes into
    $\zp G = \ve_H \zp G \op (1 - \ve_H) \zp G$ and the canonical epimorphism $\pi: \zp G \twoheadrightarrow \zp \ol G$ induces an isomorphism
    $\pi: \ve_H \zp G \simeq \zp \ol G$. But by the definition of $H$, we have $L_{S_{\infty}}(r,\chi) = 0$ if $H$ is not
    contained in the kernel of the irreducible character $\chi$.
    When $r=0$ this means that $SKu(L/K) = \ve_H \cdot SKu(L/K)$ which identifies with $SKu(L^H/K)$ via $\pi$. Using $\mu_{1-r}(L)^H = \mu_{1-r}(L^H)$,
    a similar observation holds in the case $r<0$.
    \end{proof}

    \begin{rem}
    Note that Proposition \ref{reduction-step} in particular applies when $G$ is nilpotent.
    \end{rem}

    \section{Integrality of Stickelberger elements} \label{sec:int-results}
    The aim of this section is to prove Theorem \ref{SKu-MaxOrd} and Theorem \ref{Stickelberger-int}.

    \subsection{Admissible sets of places}
    If $p$ is a prime, we denote by $S_p$ the set of $p$-adic places of $K$. We now introduce the following terminology.

    \begin{defi}
    If  $r=0$, we will say that $S$ and $T$ are  $(p,0)$\textbf{-admissible} if the following conditions are satisfied:
    \ben[(i)]
        \item
        The union of $S$ and $T$ contains all non-$p$-adic ramified primes, i.e.~$S\ram \sm(S\ram \cap S_p) \subseteq S \cup T$,
        \item
        $S$ contains all wildly ramified primes in $S_p$,
        \item
        $S$ contains the set $S_{\infty}$ of all archimedean primes,
        \item
        $S \cap T = \emptyset$,
        \item
        $E_S^{T_{\nr}}$ is torsionfree, where $T_{\nr}$ denotes the set of all unramified primes in $T$.
    \een
    If $r<0$, we will say that $S$ and $T$ are $(p,r)$\textbf{-admissible} if $Hyp(S,T)$ is satisfied.
    \end{defi}

    Note that $S$ and $T$ are in fact $(p,r)$-admissible for all primes $p$ and all $r \leq 0$ if $Hyp(S,T)$ is satisfied.

    \subsection{Integrality of Stickelberger elements}
    Recall that a finite group $G$ is called {\it monomial} if every irreducible character of $G$ is induced by a linear character.
    Examples of monomial groups are nilpotent groups \cite[Theorem 11.3]{CR-I} and, more generally, supersolvable groups
    \cite[Chapter 2, Corollary 3.5]{Weinstein}. For more information concerning monomial groups we refer the reader to \cite[Chapter 2]{Weinstein}.\\

    Now assume that $L/K$ is a finite Galois extension of number fields with Galois group $G$, where $G$ decomposes
    as $G = H \times C$ with $H$ monomial and $C$ abelian.
    As in the introduction let $\mc M(H)$ (resp.~$\mc M_p(H)$) be a maximal order in $\Q H$ (resp.~$\qp H$) containing $\Z H$ (resp.~$\zp H$).
    Then we may view $\mc M(H)[C]$ (resp.~$\mc M_p(H)[C]$) as an order in $\Q G$ (resp.~$\qp G$) and we have the following
    more general version of Theorem \ref{Stickelberger-int}.

    \begin{theo} \label{Stickelberger-int-II}
    Let $L/K$ be a finite Galois extension of number fields with Galois group $G = H \times C$, where $H$ is monomial and $C$ is abelian.
    Let $p$ be a prime and $r \in \Z_{\leq 0}$. If $S$ and $T$ are two finite sets of primes of $K$ which are $(p,r)$-admissible, then
    $$\theta_S^T(r) \in \zeta(\mc M_p(H)[C]) = \zeta(\mc M_p(H))[C].$$
    In particular, if $Hyp(S,T)$ is satisfied, we have
    $$\theta_S^T(r) \in \zeta(\mc M(H)[C]) = \zeta(\mc M(H))[C].$$
    \end{theo}

    \begin{proof}
    We first assume that $G = C$ is abelian. Then $\zeta(\mc M(H))[C] = \Z G$ and the assertion follows easily from Theorem \ref{BCDR-theo-0}
    if $r=0$ and from Theorem \ref{BCDR-theo-r} if $r<0$ as long as $Hyp(S,T)$ is satisfied. We are left with the case, where $r=0$ and
    $S$ and $T$ are $(p,0)$-admissible. We claim that $\theta_S^T(0)$ lies in $SKu_p(L/K)$ and hence Theorem \ref{BCDR-theo-0} again implies
    the desired result. To see this, we write $\theta_S^T(0)$ as
    $$\theta_S^T(0) = \de_{T_{\nr}}(0) \cdot \prod_{\fp \in T \sm T_{\nr}} \de_{\fp}(0) \prod_{\fp \in S \sm S_{\infty}} (1- \ve_{\fp} \frob\me) \cdot L(0)^{\sharp}.$$
    The set $T_{\nr}$ satisfies $Hyp(T_{\nr}, S\ram \cup S_{\infty})$ by condition (v)
    such that $\de_{T_{\nr}}(0)$ lies in $\fA_{S\ram \cup S_{\infty}}$.
    Let $\fp \in T \sm T_{\nr}$ and let $q \in \Z$ be the rational prime below $\fp$.
    By local class field theory \cite[Chapter XV, \S 2]{local-fields}, the local units at $\fp$ surject under the reciprocity map onto $\ip$.
    The subgroup of principal units is mapped onto the $q$-Sylow subgroup of $\ip$. As the factor group of the local units modulo
    the principal units has order $N(\fp)-1$,
    the ramification index
    $e_{\fp} := |\ip|$ divides $N(\fp)-1$ if $q=p$, and still up to a power of $q$ if $q \not=p$. Hence
    $$\de_{\fp}(0) = 1 - \ve_{\fp} \frob\me N(\fp) = 1 - \ve_{\fp} \frob\me - \frob\me \frac{N(\fp)-1}{e_{\fp}} N_{\ip} \in \zp \otimes U_{\fp}.$$
    For the tamely ramified primes above $p$ the element
    $$e_{\fp} = (e_{\fp} - N_{\ip})(1 - \ve_{\fp}\frob\me) + N_{\ip} \in U_{\fp}$$
    lies in $\zpmal$, since $p \nmid e_{\fp}$. Therefore, we get $\zp \otimes U_{\fp} = \zp \gp$ in this case.
    Finally, we obviously have $(1 - \ve_{\fp} \frob\me) \in U_{\fp}$ for the primes $\fp \in S \sm S_{\infty}$. Putting all this together we find
    that $\theta_S^T(0)$ belongs to $SKu_p(L/K)$ as desired.\\

    We now treat the general case, where $G = H \times C$. Since we have to deal with Stickelberger elements corresponding to
    various subextensions of $L/K$, we will write $\theta_S^T(L/K,r)$ for $\theta_S^T(r)$ for clarity. Each irreducible character of $G$ may be written
    as $\chi \cdot \la$, where $\chi \in \irr (H)$  and $\la \in \irr (C)$. We have the following decomposition
    $$\zeta(\qp [H \times C]) = \bigoplus_{\chi \in \irrp(H) / \sim} \qp (\chi)[C],$$
    where the sum runs over all ($\cp$-valued) irreducible characters of $H$ modulo Galois action and $\qp(\chi) := \qp(\chi(h)| h \in H)$.
    We fix an irreducible character $\chi$ of $H$. Then the image of $\theta_S^T(L/K,r)^{\sharp}$ in the $\chi$-component of the above
    decomposition is given by
    \beq \label{in_qpchi}
        \sum_{\la \in \irrp(C)} L_S^T(L/K,\chi \cdot \la,r) e_{\la} \in \qp (\chi)[C]
    \eeq
    and we wish to show that it actually lies in $\zp (\chi)[C]$, where $\zp(\chi)$ denotes the ring of integers in $\qp(\chi)$.
    Since $H$ is monomial, there is a subgroup $U$ of $H$ and a linear character $\psi$ of $U$ such that $\chi$
    is induced by $\psi$, i.e.~$\chi = \ind_U^H \psi$. Let us denote the abelianization of $U$ by $U\ab$; as $\psi$ is linear,
    it is inflated by a character $\psi\ab$ of $U\ab$ and hence $\chi = \ind_U^H \infl_{U\ab}^U \psi\ab$. Note that
    $\psi\ab$ is a linear character and thus irreducible. Moreover, if $\la$ is an irreducible character of $C$, we have
    \begin{equation} \label{eqn:chilambda}
        \chi \cdot \la = (\ind_{U}^H \psi) \cdot \la = \ind_{U \times C}^G (\psi \cdot \la) =
        \ind_{U \times C}^G \infl_{U\ab \times C}^{U \times C} (\psi\ab \cdot \la).
    \end{equation}
    We assure the reader that the usual behavior of $S$-truncated Artin $L$-series under direct sum, induction and inflation of characters holds for our
    $T$-modified version just as well.
    Thus (\ref{eqn:chilambda}) implies that
    $$\sum_{\la \in \irrp(C)} L_S^T(L/K,\chi \cdot \la,r) e_{\la} = \sum_{\la \in \irrp(C)} L_{S'}^{T'}(L^{[U,U]}/L^{U \times C},\psi\ab \cdot \la,r) e_{\la},$$
    where $[U,U]$ denotes the commutator subgroup of $U$ and $S' = S(L^{U \times C})$ and similarly for $T'$.
    But the righthand side lies in $\zp(\psi)[C]$, since it is the $\psi\ab$-component of the Stickelberger element
    $\theta_{S'}^{T'}(L^{[U,U]}/L^{U \times C},r)^{\sharp}$ attached to the {\it abelian} subextension $L^{[U,U]}/L^{U \times C}$.
    This and (\ref{in_qpchi}) imply that
    $$\sum_{\la \in \irrp(C)} L_S^T(L/K,\chi \cdot \la,r) e_{\la} \in \qp (\chi)[C] \cap \zp(\psi)[C] = \zp(\chi)[C]$$
    as desired. In particular, if $Hyp(S,T)$ is satisfied, then the sets $S$ and $T$ are $(p,r)$-admissible for all primes $p$, and hence
    $$\theta_S^T(r) \in \bigcap_p \left( \zeta(\mc M_p(H))[C] \cap \zeta(\Q H)[C] \right) = \zeta(\mc M(H))[C].$$
    \end{proof}

    \begin{rem}
        In the case, where the abelian group $C$ is trivial, Nomura \cite{Nomura} has shown that
        the conditions on the finite sets $S$ and $T$ can be further relaxed.
    \end{rem}

    Now let $J$ be a subset of $S\ram$ and put $S_J := (S_{\infty} \cup S\ram) \sm J$. Let $K \subseteq L_J \subseteq L$ be the maximal subfield of $L$
    such that $L_J/K$ is unramified outside $S_J$. Then $L_J / K$ is a Galois extension with Galois group $G_J = G / H_J$, where $H_J = \Gal(L/L_J)$.
    We have the following stronger version of Theorem \ref{SKu-MaxOrd}.

    \begin{theo} \label{SKu-MaxOrd-II}
    Let $L/K$ be a Galois extension of number fields with Galois group $G$ and let $J$ be a subset of $S\ram$. If the Galois group
    $G_J$ of the subextension $L_J / K$ is monomial, then
    $$\prod_{\fp \in J} \nr(N_{\ip}) \cdot \theta_{S_J}^T(r) \in \zeta(\mc M(G)),$$
    whenever $r \leq 0$ and $Hyp(S\ram \cup S_{\infty}, T)$ is satisfied. In particular,
    $$SKu'(L/K) \subseteq \zeta(\mc M(G))$$
    if $G$ is monomial.
    \end{theo}

    \begin{proof}
    Since $H_J$ is normal in $G$, the idempotent $|H_J|\me N_{H_J}$ is central in $\Q G$ and lies in $\mc M(G)$.
    If $\chi$ is an irreducible character of $G$, the $\chi$-component of $\prod_{\fp \in J} \nr(N_{\ip})$ is zero
    if $H_J$ is not contained in the kernel of $\chi$. Hence we have an equality
    \beq \label{aux-equ}
        \prod_{\fp \in J} \nr(N_{\ip}) \cdot \theta_{S_J}^T(L/K,r) = \prod_{\fp \in J} \nr(N_{\ip}) \cdot |H_J|\me N_{H_J} \cdot \ti\theta_{S_J}^T(L_J/K,r),
    \eeq
    where $\ti\theta_{S_J}^T(L_J/K,r)$ denotes any lift of $\theta_{S_J}^T(L_J/K,r)$ in $\zeta(\mc M(G))$; note that this is possible, since
    $\theta_{S_J}^T(L_J/K,r)$ lies in $\zeta(\mc M(G_J))$ by Theorem \ref{Stickelberger-int-II} as $Hyp(S_J,T)$ is satisfied for $L_J/K$. Hence the righthand
    side of the above equation also lies in $\zeta(\mc M(G))$. The second part of the theorem is clear by the definition of $SKu'(L/K)$
    and the fact that each quotient of a monomial group is again monomial.
    \end{proof}

%

    \begin{cor}
    Let $L/K$ be an abelian Galois extension of number fields with Galois group $G$ and let $J$ be a subset of $S\ram$. Then
    $$\prod_{\fp \in J} \nr(N_{\ip}) \cdot \theta_{S_J}^T(r) \in \zg,$$
    whenever $r \leq 0$ and $Hyp(S\ram \cup S_{\infty}, T)$ is satisfied.
    \end{cor}

    \begin{proof}
    If $G$ is abelian, the righthand side of equation (\ref{aux-equ}) equals
    $$\prod_{\fp \in J} N_{\ip} \cdot |H_J|\me N_{H_J} \cdot \ti\theta_{S_J}^T(L_J/K,r) = z \cdot  N_{H_J} \cdot \ti\theta_{S_J}^T(L_J/K,r),$$
    where $z = |H_J|\me \cdot \prod_{\fp \in J} |\ip|$ is an integer. The assertion follows, since
    $\theta_{S_J}^T(L_J/K,r)$ lies in $\Z G_J$ by Theorem \ref{BCDR-theo-0} and Theorem \ref{BCDR-theo-r}.
    \end{proof}

    \begin{rem}
    These results may tempt us to state a conjecture in complete analogy to Conjecture \ref{integrality-conj} also in the case $r<0$.
    We have not done so, since the author is not aware of a convincing reason, why this should be true in general.
    \end{rem}

    \subsection{An example: Frobenius groups} \label{sec:example}
    We recall the definition and some basic facts about Frobenius groups and then use them to
    provide many examples, where we can show that Stickelberger elements indeed lie in the integrality ring.

\begin{defi}
A \textbf{Frobenius group} is a finite group $G$ with a proper nontrivial subgroup $H$
such that $H \cap gHg^{-1}=\{ 1 \}$ for all $g \in G-H$,
in which case $H$ is called a \textbf{Frobenius complement}.
\end{defi}

\begin{theo}\label{thm:frob-kernel}
A Frobenius group $G$ contains a unique normal subgroup $N$, known as the Frobenius kernel, such that
$G$ is a semidirect product $N \rtimes H$. Furthermore:
\begin{enumerate}[(i)]
\item $|N|$ and $[G:N]=|H|$ are relatively prime.
\item The Frobenius kernel $N$ is nilpotent.
\item If $\chi \in \irr(G)$ such that  $N \not \leq \ker \chi$ then $\chi= \mathrm{Ind}_{N}^{G}(\psi)$ for some $1 \neq \psi \in \irr(N)$.
\end{enumerate}
\end{theo}

\begin{proof}
For (i) and (iii) see \cite[\S 14A]{CR-I}.
For (ii) see \cite[10.5.6]{Robinson}.
\end{proof}

\begin{cor} \label{cor:monomial-Frobenius}
Suppose that $G \simeq N \rtimes H$ is a Frobenius group with monomial Frobenius complement $H$.
Then $G$ is also monomial.
\end{cor}

\begin{proof}
Let $\chi \in \irr(G)$.
If $N \leq \ker \chi$ then $\chi$ is inflated from some $\varphi \in \irr(G/N)$.
Otherwise $N \nleq \ker \chi$ and so $\chi$ is induced
from some $\psi \in \irr(N)$ by Theorem \ref{thm:frob-kernel}(iii).
The Frobenius complement $H \simeq G/N$ is monomial by assumption, and
the Frobenius kernel $N$ is nilpotent by Theorem \ref{thm:frob-kernel}(ii) and thus is monomial.
Therefore in either case transitivity of induction shows that $\chi$ is induced from a linear character.
\end{proof}

The following terminology has been introduced in \cite{hybrid-ETNC}.
\begin{defi}
Let $\mc{M}_{p}(G)$ be a maximal $\Z_{p}$-order
such that $\Z_{p}[G] \subseteq \mc{M}_{p}(G) \subseteq \Q_{p}[G]$
and let $N$ be a normal subgroup of $G$.
Define the $N$-\textbf{hybrid order} of $\Z_{p}[G]$ and $\mc{M}_{p}(G)$ to be
$\mc{M}_{p}(G,N)=\Z_{p}[G] \ve_{N} \oplus \mc{M}_{p}(G)(1-\ve_{N})$.
We say that $\Z_{p}[G]$ is $N$-\textbf{hybrid} if $\Z_{p}[G] =  \mc{M}_{p}(G,N)$
for some choice of $\mc{M}_{p}(G)$.
\end{defi}

\begin{theo}\label{thm:Frobenius-integrality}
Let $L/K$ be a finite Galois extension of number fields  with $\Gal(L/K) \simeq G \times A$, where
$G \simeq N \rtimes H$ is a Frobenius group and $A$ is abelian.
Suppose that the Frobenius complement $H$ is abelian.
Then for every prime $p \nmid |N|$ and for every $(p,r)$-admissible sets $S$ and $T$ we have
    \[
        \theta_{S}^{T}(r) \in \mathcal{I}_{p}(G \times A) = \zeta(\Z_{p}[G \times A]).
    \]
\end{theo}

\begin{proof}
We first observe that $(G \times A) / N \simeq H \times A$ is abelian.
Thus $N$ contains the commutator subgroup $G'$ of $G \times A$ and so $|G'|$ is not divisible by $p$.
Hence $\mathcal{I}_{p}(G \times A) = \zeta(\Z_{p}[G \times A])$ by Proposition \ref{prop:best-denominators}.
As $H$ is abelian, it is monomial and hence $G$ is also monomial by Corollary \ref{cor:monomial-Frobenius}.
Thus Theorem \ref{Stickelberger-int-II} implies that
\begin{equation}\label{eq:theta-contained}
        \theta_{S}^{T}(r) \in \zeta(\mc{M}_{p}(G))[A].
\end{equation}
The group ring $\Z_{p}[G]$ is $N$-hybrid by \cite[Proposition 2.13]{hybrid-ETNC} and so
there is a ring isomorphism
\begin{equation}\label{eq:ring-isomorphism}
        \zeta(\Z_{p}[G \times A]) \simeq \Z_{p}[G/N \times A] \oplus (1 - \ve_N)\zeta(\mc{M}_{p}(G))[A].
 \end{equation}
If we compare (\ref{eq:theta-contained}) and (\ref{eq:ring-isomorphism}), we see that it suffices to show
that $\theta_{S}^{T}(r) \ve_{N}$ belongs to $\zeta(\Z_{p}[G \times A])\ve_{N} \simeq \Z_{p}[G/N \times A]$.
However, $G/N \times A$ is abelian and $\theta_{S}^{T}(r) \ve_{N}$ naturally identifies with the corresponding Stickelberger element attached to the
(abelian) subextension $L^{N}/K$.
Hence the result now follows from Theorem \ref{Stickelberger-int-II} with $C = G/N \times A$ and $H=1$.
\end{proof}

For an odd prime $l$ we denote by $D_{2l}$ the dihedral group of order $2l$.
If $q = l^n$ is a prime power, we let $\Aff(q)$ be the group of affine transformations on $\F_q$, the finite field with $q$ elements.
Hence $\Aff(q)$ is isomorphic to a semidirect product $\F_q \rtimes \F_q\mal$ with the natural action.
In particular, we have $\Aff(3) \simeq D_6 \simeq S_3$, the symmetric group on three letters.
Note that $D_{2l}$ and $\Aff(q)$ are Frobenius groups, and that in both cases
the Frobenius kernel coincides with the commutator subgroup.

\begin{cor}
Let $l$ be an odd prime.
Let $L/K$ be a finite Galois extension of number fields with $\Gal(L/K) \simeq G \times A$, where $A$ is abelian and $G$ is isomorphic to either $D_{2l}$ or $\Aff(q)$ where $q=l^{n}$ for some $n$.
Then for every $(p,r)$-admissible sets $S$ and $T$ we have
\[
    \theta_S^T(r) \in \mc I(G \times A).
\]
\end{cor}

\begin{proof}
If $p \neq l$ is a prime, then $\theta_S^T(r) \in \mc I_p(G \times A)$ follows from Theorem \ref{thm:Frobenius-integrality}.
If $p=l$, then by \cite[Proposition 6.7]{hybrid-ETNC} (if $G \simeq \Aff(q)$) and \cite[Proposition 6.9]{hybrid-ETNC} (if $G \simeq D_{2l}$)
we have $\mathcal{I}_{p}(G) =\zeta(\mc{M}_{p}(G))$.
Then Corollary \ref{cor:I-under-abelian-product} implies that
$\mathcal{I}_{p}(G \times A) =\zeta(\mc{M}_{p}(G))[A]$ and so the result follows from
Theorem \ref{Stickelberger-int-II}.
\end{proof}

    \section{The ETNC in almost tame extensions} \label{sec:ETNC-almost}

    \subsection{The conjecture}
    Let us fix a finite Galois extension $L/K$ of number fields with Galois group $G$ and a finite
    set $S$ of places of $K$ which contains $S_{\ram} \cup S_{\infty}$.
    Let $\De S$ be the kernel of the augmentation map
    $\Z S(L) \twoheadrightarrow \Z$ which maps each $\fP \in S(L)$ to $1$ and let
    $$\la_S: \R \otimes E_S \to \R \otimes \De S, \quad u \mapsto - \sum_{\fP \in S(L)} \log |u|_{\fP} \fP$$
    be the negative of the usual Dirichlet map. Note that $\la_S$ is in fact an isomorphism of $\R G$-modules.
    Furthermore, let $\tau_S \in \Ext_G^2(\De S, E_S)$ be Tate's canonical class (cf.~\cite{Tate-tori});
    then $\tau_S$ is given by a $2$-extension
    $$E_S \into A \rightarrow B \twoheadrightarrow \De S,$$
    where $A$ and $B$ are finitely generated cohomologically trivial $\Z G$-modules. Thus we may view
    $A \to B$ as a perfect complex with $A$ placed in degree $0$ and the pair $(A \to B, \la_S\me)$
    is a trivialized complex.
    In \cite{Burns_equivariantI} the author defines the following element of $K_0(\zg,\R)$:
    $$T\Om(L/K,0) := \psi_G^{\ast}(\chi_{\zg,\R}(A \to B, \la_S\me) + \hat\partial_G(L_S^{\ast}(0)^{\sharp})).$$
    Here, $\psi_G^{\ast}$ is a certain involution on $K_0(\zg,\R)$ which is not important for our purposes, since we
    will be only interested in the nullity of $T\Om(L/K,0)$.
    In fact, the ETNC for the motive $h^0(\Spec(L))$ with coefficients in $\zg$ in this context simply asserts the following.

    \begin{con}
    The element $T\Om(L/K,0) \in K_0(\zg,\R)$ is zero.
    \end{con}

    Note that this statement is also equivalent to the Lifted Root Number Conjecture
    formulated by Gruenberg, Ritter and Weiss \cite{GRW} (cf.~\cite[Theorem 2.3.3]{Burns_equivariantI}).
    By \cite[Theorem 2.2.4]{Burns_equivariantI} one knows that $T\Om(L/K,0)$ lies in $K_0(\zg,\Q)$ if and only if Stark's conjecture holds
    for all irreducible characters of $G$.
    In this case the ETNC decomposes into local conjectures at each prime $p$ by means of the isomorphism
    $$K_0(\zg, \Q) \simeq \bigoplus_{p \nmid \infty} K_0(\zp G, \qp).$$
    Let $T\Om(L/K,0)_p$ be the image of $T\Om(L/K,0)$ in $K_0(\zp G, \qp)$.
    If we further assume that $T\Om(L/K,0)_p$ is torsion, then it is well known (see \cite[\S 8]{GRW} for example) that
    $T\Om(L/K,0)_p$ vanishes if and only if $T\Om(L'/K',0)_p$ vanishes for all intermediate Galois extensions $L'/K'$
    whose Galois group is $p$-elementary, i.e.~$G$ is the direct product of a $p$-group and a cyclic group of order prime to $p$.
    We will prove an analogous result on minus parts in the next paragraph.

    \subsection{A reduction step on minus parts}
    Let $L/K$ be a Galois CM-extension with Galois group $G$ and let $p$ be an odd prime.
    Let $j \in G$ denote complex conjugation.
    There are canonical isomorphisms
    \begin{eqnarray*}
      K_0(\zpg, \qp) & \simeq & K_0(\zpg_+, \qp) \op K_0(\zpg_-, \qp)\\
      DT(\zpg) & \simeq & DT(\zpg_+) \op DT(\zpg_-).
    \end{eqnarray*}
    Moreover, we naturally have $\zpg_+ \simeq \zp G^+$, where $G^+ := G / \langle j \rangle$ is the Galois group
    of the extension $L^+ / K$ of totally real fields.

    Since Stark's conjecture is known for odd characters \cite[Theorem 1.2, p.~70]{Tate-Stark}, the element $T\Om(L/K,0)$ has
    a well defined image $T\Om(L/K,0)_p^-$ in $K_0(\zpg_-, \qp)$. In the proof of the following proposition we will also view $T\Om(L/K,0)_p^-$
    as an element in $K_0(\zpg, \qp)$ or rather $DT(\zpg)$ by requiring that its plus part is trivial.

    \begin{prop} \label{prop:ETNC-reduction}
      Let $L/K$ be a Galois CM-extension with Galois group $G$ and let $p$ be an odd prime.
      Assume that $T\Om(L/K,0)_p^-$ belongs to $DT(\zpg_-)$.
      Then $T\Om(L/K,0)_p^-$ vanishes if and only if $T\Om(L'/K',0)_p^-$ vanishes for all intermediate
      Galois CM-extensions $L'/K'$ whose Galois group is either $p$-elementary or a direct product of
      a $p$-elementary group and a cyclic group of order $2$ (generated by $j$).
    \end{prop}

    \begin{proof}
      Let $L'/K'$ be an \emph{arbitrary} intermediate Galois extension of $L/K$ with Galois group $H$. Then there are subgroups $G_1$ and $G_2$
      of $G$ with $G_2$ normal in $G_1$ such that $H \simeq G_1 / G_2$. There are canonical restriction and quotient maps
      \[
        DT(\zpg) \stackrel{\res^G_{G_1}}{\lto} DT(\zp G_1) \stackrel{\quot^{G_1}_H}{\lto} DT(\zp H).
      \]
      We denote the image of $T\Om(L/K,0)_p^-$ in $DT(\zp H)$ by $w_H$. Functoriality of $T\Om(L/K,0)$ \cite[Proposition 2.1.4]{Burns_equivariantI}
      implies that $w_H = T\Om(L'/K',0)_p^-$ whenever $L'/K'$ is a CM-extension. If $w_G = T\Om(L/K,0)_p^-$ vanishes,
      then clearly $w_H = 0$ for all subquotients $H$ of $G$.
      Conversely, \cite[Proposition 9]{GRW} says that $w_G = 0$ if $w_H = 0$ whenever $H$ is a $p$-elementary group.
      Fix such a $p$-elementary subquotient $H$. If $j$ lies in $G_2$, then $L'$ is totally real
      and $w_H = \quot^{G_1^+}_H (\res^{G^+}_{G_1^+} w_{G^+})$, where $G_1^+ := G_1 / \langle j \rangle$.
      However, $w_{G^+}$ is trivial and thus also $w_H = 0$. Hence we may assume that $j \not\in G_2$
      so that $L'$ is a CM-field. If $j$ lies in $G_1$, then $K'$ is totally real and $w_H = 0$ by assumption. If $j \not\in G_1$,
      we let $\ti G_1$ be the minimal subgroup of $G$ that contains $G_1$ and $j$. As $j$ is central in $G$, we have
      $\ti G_1 = G_1 \times \langle j \rangle$ and $G_2$ is still normal in $\ti G_1$. Then $\ti H := \ti G_1 / G_2 \simeq H \times \langle j \rangle$
      is a subquotient of $G$ that corresponds to an intermediate CM-extension. Moreover, we have $w_{\ti H} = 0$ by assumption
      and thus $w_H = \res^{\ti H}_H (w_{\ti H}) = 0$ as desired.
    \end{proof}

    \subsection{A reformulation in terms of Fitting invariants}
    We will say that the CM-extension $L/K$ is {\it almost tame} above $p$ if $j$ lies in $\gp$ for every prime
    $\fp$ of $K$ above $p$ which is wildly ramified in $L/K$.\\

    We have the following relation to the integrality conjecture \ref{integrality-conj}
    (cf.~\cite[proof of Theorem 5.1 and Corollary 5.6]{ich-stark}):
    \begin{theo} \label{ETNC_implies_int}
        Let $p$ be an odd prime and let $L/K$ be a Galois CM-extension. Assume that $T\Om(L/K,0)_p^-$ vanishes.
        If the $p$-part of the roots of unity
        of $L$ is a cohomologically trivial $G$-module or if $L/K$ is almost tame above $p$, then the $p$-part of
        Conjecture \ref{integrality-conj} holds, i.e.~$SKu_p(L/K) \subseteq \mc I_p(G)$.
    \end{theo}

    Thus the main result of this section (Theorem \ref{ETNC-minus-II} below) may be seen as a partial converse of Theorem \ref{ETNC_implies_int}.\\

    Now let $T$ consist of a prime $\fp_0 \nmid p$ and all finite places of $K$ which ramify in $L/K$ and do not lie above $p$;
    we may choose $\fp_0$ such that $E_S^{T_{\nr}}$ is torsionfree for every finite set $S$ of places of $K$ which contains $S_{\infty}$
    and is disjoint to $T$. Of course, $T_{\nr}$ consists of the single prime $\fp_0$.
    We denote the set of all wildly ramified primes above $p$ by $S_{p,w}$ and put $S_1 := S_{p,w} \cup S_{\infty}$.
    In particular, the sets $S_1$ and $T$ are $(p,0)$-admissible.

    \begin{theo} \label{ETNC_equiv}
        Let $p$ be an odd prime and $L/K$ a Galois CM-extension which is almost tame above $p$. Then the following are equivalent:
        \ben[(i)]
            \item
            $T\Om(L/K,0)_p^- = 0$;
            \item
            $\Fitt_{\zpg_-} (A_L^T(p)) = [ \langle \theta_{S_1}^T(0) \rangle ]_{\nr(\zpg_-)}$;
            \item
            $\Fitt_{\zpg_-} (A_L^T(p)) \subseteq [ \langle \theta_{S_1}^T(0) \rangle ]_{\nr(\zpg_-)}$.
        \een
    \end{theo}

    \begin{proof}
    First note that $A_L^T(p)$ is a cohomologically trivial $G$-module by \cite[Theorem 1]{ich-tame} such that
    the projective dimension of $A_L^T(p)$ as a $\zpg_-$-module is at most $1$.
    Thus $\Fitt_{\zpg_-} (A_L^T(p))$
    is well defined.
    That (i) and (ii) are equivalent is just a reformulation of \cite[Theorem 2]{ich-tame} in terms of Fitting invariants.
    To see this let $E$ be a finite Galois extension of $\qp$ with Galois group $\Ga$ such that every odd representation of $G$
    has a realization over $E$. Let $R_p(G)^-$ denote the subring of $R_p(G)$ generated by odd characters.
    Then remark \ref{rem:rephom-vs-fittgen} has an obvious analogue on minus parts, and in the notation of \S \ref{sec:Hom-description}
    we see that $\theta_{S_1}^T(0) \in \zeta(\qp G_-)\mal$
    corresponds to the representing homomorphism $f_{\theta_{S_1}^T(0)} \in \Hom_{\Ga}(R_p(G)^-, E\mal)$
    which on irreducible odd characters $\chi$ is given by
    \[
        f_{\theta_{S_1}^T(0)}: \chi \mapsto L_{S_1}^T(0, \check \chi).
    \]
    However, in the notation of \cite{ich-tame}
    we have $f_{\theta_{S_1}^T(0)} = \Theta_{S_1}^T$.

    Clearly (ii) implies (iii) and we are left with showing the converse.
    Let $z \in (\qp G_-)\mal$ be a generator of $\Fitt_{\zpg_-} (A_L^T(p))$ and let $f_z \in \Hom_{\Ga}(R_p(G)^-, E\mal)$
    be the corresponding representing homomorphism. Write
    $$z = \sum_{\chi} z_{\chi} e_{\chi} \in \zeta(E G_-) = \oplus_{\chi} E e_{\chi},$$
    where the sum runs through all odd irreducible ($\cp$-valued) characters of $G$.
    Then $f_z(\chi) = z_{\chi}$ for all odd irreducible $\chi$ by (\ref{eqn:f_z}), and by \cite[Proposition 5]{ich-tame} we have
    $$\prod_{\chi} z_{\chi}^{\chi(1)} \sim_p |A_L^T(p)|,$$
    where $\sim_p$ means ``equality up to a $p$-adic unit''.
    But by \cite[Proposition 4]{ich-tame}, we also have
    $$\prod_{\chi} (L^T_{S_1}(0,\check \chi))^{\chi(1)} \sim_p |A_L^T(p)|$$
    so that $z$ is also a generator of $[ \langle \theta_{S_1}^T(0) \rangle ]_{\nr(\zpg_-)}$ by \cite[Proposition 5.4]{ich-Fitting}.
    Hence (iii) implies (ii) and we are done.
    \end{proof}

    \subsection{Exceptional primes}
    For a natural number $n$ let $\zeta_n$ be a primitive $n$th root of unity and let us denote the normal closure of $L$ over $\Q$ by $L\hcl$;
    note that $L\hcl$ is again a CM-field.

    \begin{defi} \label{defi:exceptional}
    We will call a prime $p$ \textbf{exceptional} for $L/K$ if at least
    one of the following holds:
    \ben[(i)]
        \item
        $p = 2$,
        \item
        there is a prime $\fp$ in $K$ above $p$ which ramifies wildly in $L$ and $j \not\in \gp$, i.e.~$L/K$ is not almost tame above $p$,
        \item
        $L\hcl \subseteq (L\hcl)^+(\zeta_p).$
    \een
    \end{defi}

    \begin{rem} \label{finite-exceptionals}
    \ben
        \item
        Note that there are only finitely many exceptional primes, since such a prime has to ramify in $L\hcl / \Q$ or equals $2$.
        \item
        If $p$ is non-exceptional, then the negation of (ii) ensures that the $p$-minus ray class group $A_L^T(p)$ is a cohomologically trivial $G$-module.
        We have already used this in the proof of Theorem \ref{ETNC_equiv}.
        In fact, there is a second technical point, where (ii) is needed. We will indicate this in the course of the proof.
        \item
        The negation of (iii) will permit us to adjust a descent method introduced
        by Wiles \cite{Wiles-Brumer} and further developed by Greither \cite{Gr-Brumer}
        to the non-abelian situation. In particular, it is needed in the proof of Lemma \ref{fields-exist}.
    \een
    \end{rem}

    \begin{lem} \label{lem:exceptional-intermediate}
      Let $L/K$ be a Galois CM-extension and let $p$ be a prime. If $p$ is non-exceptional for $L/K$, then it is non-exceptional for every
      intermediate Galois CM-extension $L'/K'$.
    \end{lem}

    \begin{proof}
      Let $p$ be non-exceptional for $L/K$ and let $L'/K'$ be an intermediate Galois CM-extension with Galois group $H$.
      Clearly $p \not=2$ and so (i) does not hold.
      Suppose that there is a prime $\fp'$ in $K'$ above $p$ such that $\fp'$ is wildly ramified in $L'$. Then $\fp := \fp' \cap K$
      is a prime in $K$ above $p$ which is wildly ramified in $L$. Thus $j \in G_{\fP}$ for every prime $\fP$ in $L$ above $\fp$
      as $p$ is non-exceptional for $L/K$. We write $H = G_1 / G_2$, where $G_1$ and $G_2$ are subgroups of $G$ with $G_2$ normal in $G_1$.
      As $K'$ is totally real, we have $j \in G_1$ and thus $j \in G_{1, \fP} = G_1 \cap G_{\fP}$.
      Let $\fP'$ be the prime in $L'$ below $\fP$. Then the natural surjection $G_1 \twoheadrightarrow H$ maps $j \in G_1$ to $j \in H$ and
      $G_{1,\fP}$ onto $H_{\fP'}$. So $j \in H_{\fP'}$ and (ii) for $L'/K'$ does not hold either.

      Now suppose that (iii) holds for $p$ and $L'/K'$. Then we have
      \[
        (L')\hcl \subseteq ((L')\hcl)^+(\zeta_p) \subseteq (L\hcl)^+(\zeta_p)
      \]
      and thus also
      \[
        (L')\hcl (L\hcl)^+ \subseteq (L\hcl)^+(\zeta_p).
      \]
      However, the field $(L')\hcl$ is a CM-field and therefore not contained in $(L\hcl)^+$. Hence the first inclusion in
      \[
        (L\hcl)^+ \subseteq (L\hcl)^+ (L')\hcl \subseteq L\hcl
      \]
      is proper. We find that $(L\hcl)^+ (L')\hcl = L\hcl$ is contained in $(L\hcl)^+(\zeta_p)$, a contradiction.
    \end{proof}

    \subsection{The main theorem}
    We now prove the following theorem which is Theorem \ref{ETNC-minus} of the introduction.

    \begin{theo} \label{ETNC-minus-II}
    Let $L/K$ be a Galois CM-extension of number fields with Galois group $G$ and let $p$ be a non-exceptional prime.
    If the Iwasawa $\mu$-invariant attached to $L$ and $p$ vanishes, then
    the $p$-minus part of the ETNC for the pair $(h^0(\Spec(L)), \Z G)$ is true.
    \end{theo}

    \begin{proof}
    We first observe that we may assume that $G$ is monomial. In fact, we already know
    that $T\Om(L/K,0)_p^-$ is torsion, since the strong Stark conjecture holds by \cite[Corollary 2]{ich-tame}.
    Then Proposition \ref{prop:ETNC-reduction} implies that we may assume that $G$ is either $p$-elementary
    or a direct product of a $p$-elementary group and a cyclic group of order $2$.
    In both cases, the group $G$ is nilpotent and thus monomial.
    Note that $p$ is still a non-exceptional prime by Lemma \ref{lem:exceptional-intermediate}.
    We will henceforth assume that $G$ is monomial. In particular, we may use
    the results of \cite{ich-tameII}, where the identity (\ref{eqn:values-padic-complex})
    is implicitly assumed to hold.

    Let $L_{\infty}$ and $K_{\infty}$ be the cyclotomic
    $\zp$-extensions of $L$ and $K$, respectively. We denote the
    Galois group of $K_{\infty}/K$ by $\Ga_K$. Hence $\Ga_K$ is
    isomorphic to $\zp$, and we fix a topological generator $\ga_K$.
    Accordingly, we set $\Ga_L = \Gal(L_{\infty}/L)$ with
    a topological generator $\ga_L$.
    Furthermore, we denote the $n$-th layer in the cyclotomic
    extension $L_{\infty}/L$ by $L_n$ such that $L_n/L$ is cyclic of
    order $p^n$.  We put
    $$\mc X_T^- := \lim_{\leftarrow}  A_{L_n}^{T}(p) .$$
    We denote the Galois group of $L_{\infty}/K$ by $\mc G$; hence
    $\mc G = H \rtimes \Ga,$
    where $H$ is a subgroup of $G$ and $\Ga$ is topologically generated by a preimage $\ga$ of $\ga_K$ under the canonical
    epimorphism $\mc G \twoheadrightarrow \mc G / H = \Ga_K$. We denote the Iwasawa algebra $\zp[[\mc G]]$ by $\La(\mc G)$.
    Then $\mc X_T^-$ is a finitely generated $R$-torsion $\La(\mc G)_- := \La(\mc G) / (1+j)$-module,
    where $R := \zp[[\Ga']]$ with $\Ga'\simeq \zp$ central in $\mc G$. Note that $R$ is isomorphic to $\zp[[T]]$,
    the power series ring in one variable over $\zp$.\\

    The vanishing of the Iwasawa $\mu$-invariant implies that the $\mu$-invariant of $\mc X_T^-$ also vanishes; hence the projective dimension
    of $\mc X_T^-$ over $\La(\mc G)_-$ is at most $1$ by \cite[Proposition 4.1]{ich-tameII}. Then by
    \cite[Lemma 6.2]{ich-Fitting} the $\La(\mc G)_-$-module $\mc X_T^-$ admits a quadratic presentation
    and thus $\Fitt_{\La(\mc G)_-}(\mc X_T^-)$ is well defined. This Fitting invariant is computed via the
    equivariant Iwasawa main conjecture (which is a theorem under our current hypotheses by Ritter and Weiss \cite{EIMC-theorem}
    and Kakde \cite{Kakde-mc}) in \cite[Theorem 4.4]{ich-tameII}. Since the precise statement of this theorem would force us to introduce
    a lot of further notation, we only state the following consequence \cite[Lemma 6.3]{ich-tameII}
    of this theorem which will be sufficient for our purposes:
    \beq \label{EIMC-consequence}
        \Fitt_{\zpg_-} ((\mc X^-_T)_{\Ga_L}) = [ \langle \theta_{S_p}^T(0) \rangle ]_{\nr(\zpg_-)},
    \eeq
    where $(\mc X^-_T)_{\Ga_L}$ denotes the $\Ga_L$-coinvariants of $\mc X^-_T$.\\

    We now adopt a method originally introduced by Wiles \cite{Wiles-Brumer} and further developed (in an equivariant way) by
    Greither \cite{Gr-Brumer} and the author \cite{ich-tame}. In fact, the following is carried out in some detail in \cite{ich-tameII};
    but there, the full integrality conjecture \ref{integrality-conj} is assumed to hold for an infinite class of field extensions.
    Since we only will use our results established in \S \ref{sec:int-results}, we have to take care if everything still works.

    \begin{lem} \label{fields-exist}
        Let $N>0$ be a natural number. Then there are infinitely many primes $r \in \Z$ such that
        \ben[(i)]
            \item
            $r \equiv 1 \mod p^N$.
            \item
            $j \in G_{\fR}$ for all primes $\fR$ in $L$ above $r$.
            \item
            The Frobenius automorphism $\Frob_p$ at $p$ in $\Gal(\Q(\zeta_r) / \Q)$ generates $\Gal(k_r / \Q)$,
            where $k_r$ denotes the unique subfield of $\Q(\zeta_r)$ of degree $p^N$ over $\Q$.
        \een
    \end{lem}
    \begin{proof}
        This is \cite[Lemma 6.5]{ich-tameII}. But the proof of \cite[Proposition 4.1]{Gr-Brumer} carries over unchanged to the present situation.
    \end{proof}

    Let $N \in \N$ be a positive integer and choose a prime $r$ as in Lemma \ref{fields-exist} which does not ramify in
    $L\hcl / \Q$. We put $L' := Lk_r$, $K' = Kk_r$ and $G' = \Gal(L'/K) = G \times C_N$, where
    $C_N \simeq \Gal(k_r / \Q)$ is cyclic of order $p^N$, generated by $\Frob_p$.
    Note that $p$ is still a non-exceptional prime for $L'/K$ and that $G'$ is also monomial.
    Moreover, we define
    $T' := T \cup S_r$, where $S_r$ denotes the set of places in $K$ above $r$.
    We have an exact sequence
    \[
        (\zp \otimes (\mc{O}_{L'} / \prod_{\fR |r} \fR)\mal)^- \into A_{L'}^{T'}(p) \twoheadrightarrow A_{L'}^{T}(p).
    \]
    Since $j \in G_{\fR}$ for all $\fR \mid r$, we may conclude
    as in \cite[p.~28]{ich-tame} to deduce that the leftmost term is trivial. We obtain an isomorphism
    \begin{equation} \label{eqn:ray-class-iso}
        A_{L'}^{T'}(p) \simeq A_{L'}^T(p)
    \end{equation}
    and hence $ A_{L'}^T(p)$ is cohomologically trivial as $G'$-module by \cite[Theorem 1]{ich-tame}.
    As in \cite[p.~28]{ich-tame} the restriction map induces an isomorphism
    \beq \label{C_N_iso}
        (A_{L'}^T(p))_{C_N} \simeq A_L^T(p).
    \eeq
    More precisely, the cokernel of the restriction map $A_{L'}^T(p) \to A_L^T(p)$ identifies with
    a quotient of $C_N$. However, $j$ acts by conjugation and thus trivially on $C_N$. Hence the cokernel is trivial.
    The composite map
    \[
        A_{L'}^T(p) \stackrel{\res}{\lto} A_{L}^T(p) \lto A_{L'}^T(p)
    \]
    is given by the norm of the cyclic group $C_N$. As $A_{L'}^T(p)$ is cohomologically trivial, the kernel of the norm
    is precisely $(\Frob_p - 1) A_{L'}^T(p)$. This gives the desired isomorphism (\ref{C_N_iso}).

    The sets $S_1$ and $T'$ are $(p,0)$-admissible and thus
    the Stickelberger element $\theta_{S_1}^{T'}(L'/K,0)$ lies in $\zeta(\mc M_p(G))[C_N]$ by Theorem \ref{Stickelberger-int-II}.
    However, we have
    \begin{equation} \label{eqn:two-Stickelberger}
        \theta_{S_1}^{T'}(L'/K,0) = \left(\frac{1-j}{2}\de_{S_r}(0)\right) \theta_{S_1}^{T}(L'/K,0)
    \end{equation}
    and we claim that
    \begin{equation} \label{eqn:differ-by-unit}
        \frac{1-j}{2}\de_{S_r}(0) \in \nr((\zp G'_-)\mal).
    \end{equation}
    In fact, as mentioned above,
    we have
    $$(\zp \otimes (\mc{O}_{L'} / \prod_{\fR |r} \fR)\mal)^-  = 0.$$
    Hence $\frac{1-j}{2}\de_{S_r}(0)$ is a generator
    of $\Fitt_{(\zp G')_-}(0)$ and therefore  lies in $\nr((\zp G'_-)\mal)$.
    However, we have $\nr((\zp G'_-)\mal) \subseteq \mc I_p(G') \subseteq \zeta(\mc M_p(G))[C_N]$ by Lemma \ref{int-lem},
    and thus
    \beq \label{theta-int}
        \theta_{S_1}^T(L'/K,0) \in \zeta(\mc M_p(G))[C_N]
    \eeq
    for all $N$.
    We define an element $\al_p \in \mc I_p(G)$ by
    $$\al_p = \prod_{\fp \in S_p \sm S_1} \nr(1- \ve_{\fp}\frob\me)$$
    such that we have an equality $\theta_{S_1}^T(L/K,0) \cdot \al_p = \theta_{S_p}^T(L/K,0)$.
    Similarly, we define $\al_p' \in \mc I_p(G')$ such that $\theta_{S_1}^{T}(L'/K,0) \cdot \al_p' =
    \theta_{S_p}^{T}(L'/K,0)$. Now choose a second natural number $M\leq N$ and put
    $$\nu := \sum_{i=0}^{p^M-1} \Frob_p^{i p^{N-M}} \in \zp C_N \subseteq \zeta(\zp G').$$
    The following result is \cite[Lemma 6.7]{ich-tameII}.

    \begin{lem} \label{nonzerolemma}
        Let $f$ be the least common multiple of the residual degrees $f_{\fp}(K/\Q)$ of all $\fp \in S_p$.
        If $N-M \geq v_p(|G| \cdot f)$, then $|G| \cdot \al_p'$ is a nonzerodivisor in $\zeta(\zp G') / \nu$.
    \end{lem}

        We now observe that enlarging $L$ to $L'$ does not affect the vanishing of $\mu$ by \cite[Theorem 11.3.8]{NSW}.
        Choose natural numbers $M \leq N$ such that $r = r(N)$ fulfills all the above conditions
        and $N-M \geq v_p(|G| \cdot f)$, where $f$ was defined in Lemma \ref{nonzerolemma}.
        Let $\mc G' = \Gal(L'_{\infty}/K)$ and let $\mc X_{T'}^-$ be the projective limit of the minus $p$-ray class groups
        $A_{L'_{n}}^{T'}(p)$. Then
        $\mc X_{T'}^-$ has projective dimension at most one as before and the EIMC for the extension $(L'_{\infty})^+/K$
        implies the following analogue of equation (\ref{EIMC-consequence}):
        \begin{equation} \label{eqn:easy-descent}
            \Fitt_{\zp G'_-} ((\mc X^-_{T'})_{\Ga_{L'}}) = [ \langle \theta_{S_p}^{T'}(L'/K,0) \rangle ]_{\nr(\zp G'_-)}.
        \end{equation}
        For each prime $\fp$
        of $K$ let $\fP' \subset L'$ be a prime above $\fp$. By \cite[Proposition 4.7]{ich-tameII}, we have a right exact sequence
        \begin{equation} \label{eqn:important-ses}
            (\bigoplus_{\fp \in S_p} \ind_{G'_{\fP'}}^{G'} \zp )^- \to (\mc X_{T'}^-)_{\Ga_{L'}} \twoheadrightarrow A_{L'}^{T'}(p).
        \end{equation}
        Note that $(\ind_{G'_{\fP'}}^{G'} \zp)^- = 0$ whenever $\fp \in S_{p,w} = S_p \cap S_1$,
        since $j$ lies in the decomposition group $G'_{\fP'}$ in this case;
        it is here, where we use that (ii) of Definition \ref{defi:exceptional} does not hold for $p$.
        Therefore the Fitting invariant of the leftmost term is generated by $\al'_p$.
        By (\ref{eqn:two-Stickelberger}) and (\ref{eqn:differ-by-unit}) the Stickelberger elements $\theta_{S_p}^{T'}(L'/K,0)$ and
        $\theta_{S_p}^{T}(L'/K,0)$ only differ by the norm of a unit.
        Hence $\theta_{S_p}^{T}(L'/K,0)$ is also a generator of $\Fitt_{\zp G'_-}((\mc X_{T'}^-)_{\Ga_{L'}})$ by (\ref{eqn:easy-descent}).
        The above sequence (\ref{eqn:important-ses}) gives rise to the
        following inclusion of Fitting invariants (cf.~\cite[Proposition 3.5 (iii)]{ich-Fitting}):
        $$\Fitt_{\zp G'_-}\left((\bigoplus_{\fp \in S_p} \ind_{G'_{\fP'}}^{G'} \zp)^-\right) \cdot \Fitt_{\zp G'_-}(A_{L'}^{T'}(p))
        \subseteq \Fitt_{\zp G'_-}((\mc X_{T'}^-)_{\Ga_{L'}}).$$
        If we choose a generator $c'$ of $\Fitt_{\zp G'_-}(A_{L'}^{T'}(p))$, there exists $x \in \zeta(\zp G')$ such that
        $$\al'_p c' = x \cdot \theta_{S_p}^{T}(L'/K,0) = x \cdot \al'_p \theta_{S_1}^{T}(L'/K,0).$$
        It follows from Lemma \ref{int-lem} and (\ref{theta-int}) that multiplication by $|G|$ yields an equality in
        $\zeta(\zp G')$
        such that Lemma \ref{nonzerolemma} gives
        \beq \label{equ-mod-nu}
            |G| \cdot c' \equiv |G| \cdot x \cdot \theta_{S_1}^{T}(L'/K,0) \mod \nu.
        \eeq
        Let $\aug: \zp G' \to \zp G$ be the natural augmentation map.
        Since Fitting invariants behave well under base change (cf.~\cite[Lemma 5.5]{ich-Fitting}),
        the element $c := \aug (c')$ generates the Fitting invariant of $A_L^T(p)$ by (\ref{eqn:ray-class-iso}) and (\ref{C_N_iso}).
        Since $\aug (\theta_{S_1}^{T}(L'/K,0)) = \theta_{S_1}^{T}(L/K,0)$ and $\aug (\nu) = p^M$,
        the congruence (\ref{equ-mod-nu}) implies
        $$c \equiv \aug(x) \cdot \theta_{S_1}^T(L/K,0) \mod p^{M-m} \mc I_p(G),$$
        where $p^m$ is the exact $p$-power dividing $|G|$. This gives an inclusion
        $$\Fitt_{\zpg}(A_L^T(p)) \subseteq [\langle \theta_{S_1}^{T}(L/K,0) \rangle]_{\nr(\zpg)},$$
        as we may choose $M$ arbitrarily large. Now we are done
        via Theorem \ref{ETNC_equiv}.
    \end{proof}

    We end this section with giving the proofs of some of the corollaries mentioned in the introduction.
    \begin{proof}[Proof of Corollary \ref{cor_int_conj}]
    This is an immediate consequence of Theorem \ref{ETNC-minus-II} and Theorem \ref{ETNC_implies_int}.
    \end{proof}

    \begin{proof}[Proof of Corollary \ref{cor-etnc}]
    The central conjecture (Conjecture 2.4.1) of \cite{BB_at_1} states that a certain element $T\Om(L/K,1) \in K_0(\zg, \R)$ vanishes.
    By \cite[Theorem 5.2]{BB_at_1} one has an equality
    $$\psi_G^{\ast}(T\Om(L/K,0)) - T\Om(L/K,1) = T\Om^{loc}(L/K,1),$$
    and the vanishing of the righthand side is equivalent to a conjecture of Bley and Burns \cite{epsilon} by \cite[Remark 5.4]{BB_at_1}.
    But this conjecture is known if $L/K$ is at most tamely ramified by \cite[Corollary 6.3 (i)]{epsilon}.
    Finally, if we suppose that Leopoldt's conjecture holds,
    then \cite[Theorem 1.1 and Corollary 1.2]{BB_at_1_II} imply the desired relation to the ETNC for the pair $(h^0(\Spec(L))(1), \zg)$.
    However, it is sufficient for our purposes that Leopoldt's conjecture holds on minus parts and this is in fact true:

    Recall that by \cite[Theorem 10.3.6]{NSW} Leopoldt's conjecture for odd $p$ is equivalent to the assertion that the canonical homomorphism
    \[
        \Delta: \mc O_L\mal \otimes \zp \lto \prod_{\fP \in S_p(L)} \widehat{\mc O}_{\fP}\mal
    \]
    is injective, where $\widehat{\mc O}_{\fP}\mal$ denotes the pro-$p$-completion of the group of units of the local field $L_{\fP}$
    (hence $\widehat{\mc O}_{\fP}\mal$ is canonically isomorphic to the group of principal units).
    By \cite[Lemma 10.3.13 and Lemma 8.7.7]{NSW} the kernel of $\Delta$ is torsion-free. However, the minus part of
    $\mc O_L\mal \otimes \zp$ are the $p$-power roots of unity in $L$. It follows that $\Delta$ is injective on minus parts as desired.
    \end{proof}

    \section{The non-abelian Brumer-Stark conjecture} \label{sec:Brumer-Stark}

    As before, let $L/K$ be a Galois CM-extension with Galois group $G$.
    The following conjecture has been formulated in \cite{ich-stark} and is a non-abelian generalization of Brumer's conjecture.

    \begin{con} \label{Brumer}
        Let $S$ be a finite set of places of $K$ containing $S_{\ram} \cup S_{\infty}$. Then
        $\fA_S \theta_S(0) \subseteq \mathcal I(G)$ and for each $x \in \mathcal H(G)$
        we have
        $$x \cdot \fA_S \theta_S(0) \subseteq \Ann_{\zg} (\cl_L).$$
    \end{con}

    \begin{rem}
        \bit
        \item
            If $G$ is abelian, the inclusion
            $\fA_S \theta_S(0) \subseteq \mathcal I(G) = \Z G$ holds by Theorem \ref{BCDR-theo-0} and,
            since $\mathcal H(G) = \Z G$ in this case, Conjecture \ref{Brumer} recovers
            Brumer's conjecture.
        \item
            Replacing the class group $\cl_L$ by its $p$-parts $\cl_L(p)$ for each rational prime $p$,
            Conjecture \ref{Brumer} naturally decomposes into local conjectures at each prime $p$.
            Note that it is then possible to replace $\mc H(G)$ by $\mc H_p(G)$ by \cite[Lemma 1.4]{ich-stark}.
        \item
            Burns \cite{Burns-derivatives} has also formulated a conjecture which generalizes many refined Stark conjectures to the
            non-abelian situation. In particular, it implies this generalization of Brumer's conjecture
            (cf.~\cite[Proposition 3.5.1]{Burns-derivatives}).
        \eit
    \end{rem}

    For $\al \in L\mal$ we define
    $$S_{\al} :=  \{\fp \subset K: \fp | N_{L/K}(\al) \}$$
    and we call $\al$ an {\it anti-unit} if $\al^{1+j} = 1$.
    Let $\om_L := \nr (|\mu_L|)$. The following is a non-abelian generalization of the
    Brumer-Stark conjecture (cf.~\cite[Conjecture 2.6]{ich-stark}).

    \begin{con} \label{Brumer-Stark}
        Let $S$ be a finite set of places of $K$ containing $S_{\ram} \cup S_{\infty}$.
        Then $\om_L \cdot \theta_S(0) \in \mathcal I(G)$ and for each $x \in \mathcal H(G)$
        and each fractional ideal $\fa$ of $L$, there is an anti-unit $\al = \al(x,\fa,S) \in L\mal$ such that
        $$\fa^{x \cdot \om_L \cdot \theta_S(0)} = (\al)$$
        and for each finite set $T$ of primes of $K$ such that $Hyp(S \cup S_{\al},T)$ is satisfied
        there is an $\al_T \in E_{S_{\al}}^T$
        such that
        \beq \label{abelian-ersatz}
            \al^{z \cdot \de_T(0)} = \al_T^{z \cdot \om_L}
        \eeq
         for each $z \in \mc H(G)$.
    \end{con}

    \begin{rem}
    \bit
        \item
        If $G$ is abelian, we have $\mc I(G) = \mc H(G) = \Z G$ and $\om_L = |\mu_L|$.
        Hence it suffices to treat the case $x=z=1$.
        Then \cite[Proposition 1.2, p.~83]{Tate-Stark} states that the condition (\ref{abelian-ersatz})
        on the anti-unit $\al$ is equivalent to the assertion that the extension $L(\al^{1/\om_L}) / K$
        is abelian.
        \item
        As above, we obtain local conjectures for each prime $p$.
        \item
        The non-abelian Brumer-Stark conjecture (at $p$) implies the non-abelian Brumer conjecture (at $p$) by \cite[Lemma 2.12]{ich-stark}.
    \eit
    \end{rem}

    We now prove the remaining results mentioned in the introduction which are mainly concerned with the above two conjectures.

    \begin{proof}[Proof of Corollary \ref{cor-non-ab-stark}]
    Theorem \ref{ETNC-minus-II} and \cite[Theorem 5.3]{ich-stark} imply that $L/K$ fulfills the (non-abelian) strong Brumer-Stark property at $p$.
    This in turn implies (ii) by \cite[Proposition 3.8]{ich-stark} and (i) by \cite[Lemma 2.9]{ich-stark}. Since the condition $L\hcl \not\subseteq (L\hcl)^+(\zeta_p)$
    forces $\zeta_p \not\in L$, the $p$-part of the roots of unity is trivial and thus cohomologically trivial as $G$-module. Therefore
    Theorem \ref{ETNC-minus-II} and \cite[Theorem 4.1.1]{Burns-derivatives} imply (iii). Finally, as already mentioned above, the vanishing
    of $T\Om(L/K,0)$ is equivalent to the Lifted Root Number Conjecture of Gruenberg, Ritter and Weiss
    as formulated in \cite{GRW} (cf.~\cite[Theorem 2.3.3]{Burns_equivariantI}). Thus Theorem \ref{ETNC-minus-II} also implies (iv).
    \end{proof}

    \begin{proof}[Proof of Theorem \ref{non-abelian-Stickelberger}]
        It suffices to prove the Brumer-Stark conjecture at each odd prime $p$. We first assume that $p$ is unramified. As $L = L\hcl$ in this case,
        the prime $p$ is non-exceptional by Remark \ref{finite-exceptionals}. In particular,
        $L/\Q$ is almost tame above $p$ and the result follows from Corollary \ref{cor-non-ab-stark} if $\mu$ vanishes. If $\mu$ does not vanish
        (which conjecturally will never be the case) then $p \nmid [L:\Q]$ by assumption. Hence the $p$-minus part of the ETNC is equivalent
        to the $p$-part of the strong Stark conjecture for odd characters which is a theorem by \cite[Corollary 2]{ich-tame}.
        Hence the Brumer-Stark conjecture at $p$ holds by \cite[Theorem 5.2]{ich-stark}. Now assume that $p$ ramifies in $L$. Since
        $p$ is the only $p$-adic place of the rationals, we have $S_p \subseteq S\ram$, and the result follows from \cite[Corollary 4.6]{Nickel-IwasawaStark},
        where the identity (\ref{eqn:values-padic-complex}) is implicitly assumed to hold (it is only here, where we have to restrict to monomial extensions).
    \end{proof}

    \section{Negative integers} \label{sec:negative-integers}

    For completeness, we include the following result which is an easy consequence of \cite[Theorem 4.1]{ich-K-groups}
    and \cite[Corollary 2.10]{Burns-mc}.

    \begin{theo} \label{ETNC_implies_int_r}
    Let $L/K$ be a Galois extension of number fields with Galois group $G$ and let $r<0$. Assume that
    $L$ is totally real if $r$ is odd (resp.~that $L/K$ is CM if $r$ is even). If $p$ is an odd prime
    such that the $p$-part (resp.~minus $p$-part) of the ETNC for the pair $(h^0(\Spec(L))(r), \zg)$ holds,
    then the $p$-part of Conjecture \ref{integrality-conj_r} is true.
    In particular, this applies if the Iwasawa $\mu$-invariant attached
    to $L$ and $p$ vanishes.
    \end{theo}

    \begin{cor}
    Let $L/K$ be a Galois extension of number fields with Galois group $G$ and let $r<0$. Let $p$
    be an odd prime and assume that $G$ has a unique $2$-Sylow subgroup. Then the $p$-part of
    Conjecture \ref{integrality-conj_r} holds provided that the Iwasawa $\mu$-invariant
    attached to $p$ and the maximal real subfield of $L$ vanishes.
    \end{cor}

    \begin{proof}
    This immediately follows from Theorem \ref{ETNC_implies_int_r} and Proposition \ref{reduction-step}.
    \end{proof}

\noindent Andreas Nickel~~ anickel3@math.uni-bielefeld.de\\
Universit\"{a}t Bielefeld,
    Fakult\"{a}t f\"{u}r Mathematik,
    Postfach 100131,
    33501 Bielefeld,
    Germany
\end{document}